\documentclass[10pt,reqno]{amsart}
\usepackage{amsfonts}
\usepackage{amssymb,latexsym}
\usepackage{amsmath}
\usepackage{amsthm}
\usepackage{amssymb}
\usepackage{enumerate}
\usepackage{hyperref}\usepackage{xcolor}

\usepackage{array}
\usepackage{delarray}
\usepackage{amsthm}
\usepackage{mathtools}
\usepackage{hyperref}

\usepackage{cleveref}

\numberwithin{equation}{section}
\newtheorem{theorem}{Theorem}[section]
\newtheorem{lemma}[theorem]{Lemma}
\newtheorem{remark}[theorem]{Remark}
\newtheorem{definition}[theorem]{Definition}
\newtheorem{proposition}[theorem]{Proposition}

\def\XXint#1#2#3{{\setbox0=\hbox{$#1{#2#3}{\int}$ }
\vcenter{\hbox{$#2#3$ }}\kern-.6\wd0}}

\newcommand{\R}{{\mathbb R}}

\newcommand{\bfa}{\mathbf a}
\newcommand{\cL}{{\mathcal L}}

\newcommand{\al}{\alpha}

\newcommand{\de}{\delta}
\newcommand{\e}{\varepsilon}

\newcommand{\la}{\lambda}
\newcommand{\si}{\sigma}

\newcommand{\Si}{\Sigma}

\newcommand{\ti}{\times}

\newcommand{\pa}{\partial}

\newcommand{\qu}{\quad}

\newcommand{\ra}{\rightarrow}

\newcommand{\D}{\nabla}
\newcommand{\De}{\Delta}

\newcommand{\ind}{\operatorname{index}}

\newcommand{\Vol}{\operatorname{Vol}}

\newcommand{\fr}{\frac}

\newcommand{\inn}[2]{\left\langle {#1},{#2} \right\rangle}

\title{Ancient mean curvature flows with finite total curvature}

\author{Kyeongsu Choi}
\address{KC: School of Mathematics, Korea Institute for Advanced Study, 85 Hoegiro, Dongdaemun-gu, Seoul 02455, Republic of Korea}
\email{choiks@kias.re.kr}

\author{Jiuzhou Huang}
\address{JH: School of Mathematics, Korea Institute for Advanced Study, 85 Hoegiro, Dongdaemun-gu, Seoul 02455, Republic of Korea}
\email{jiuzhou@kias.re.kr}

\author{Taehun Lee}
\address{TL: School of Mathematics, Korea Institute for Advanced Study, 85 Hoegiro, Dongdaemun-gu, Seoul 02455, Republic of Korea}
\email{taehun@kias.re.kr}

\date{}

\begin{document}
\allowdisplaybreaks

\begin{abstract}
    We construct an $I$-family of ancient graphical mean curvature flows over a minimal hypersurface in $\mathbb{R}^{n+1}$ of finite total curvature with the Morse index $I$ by establishing exponentially fast convergence in terms of $|x|^2-t$. As a corollary, we show that these ancient flows have finite total curvature and finite mass drop. Moreover, one family of these flows is mean convex by a pointwise estimate.
\end{abstract}

\maketitle
%section
\section{Introduction}
A smooth family of embedded hypersurfaces $M_t\subset\mathbb{R}^{n+1}$ evolves by \textit{mean curvature flow} (MCF) if 
\begin{equation}
     F_t=-  H \nu
\end{equation}
where  $F$ is the position vector of $M_t$ and $H$, $\nu$ are the mean curvature and the outward pointing unit normal of $M_t$ at $F$, respectively. Also, if $M_t$ exists for $t\in (-\infty,T)$, then we call it \textit{ancient}.

By parabolic Liouville type theorems, ancient mean curvature flows can be characterized. See classification results for ancient flows which are noncollpased \cite{DHS, ADS19, ADS20, BC19, BC21, CDDHS22, DH24, CHH23, CHH21}, asymptotically cylindrical \cite{CHH,CHHW}, or one-sided \cite{CCMS20, CCS23}. See also construction results of ancient flows, which are, after rescaling, asymptotic to cylinders  \cite{HH}, compact shrinkers \cite{ChoiM},
asymptotically conical shrinker \cite{CCMS20}, or a certain class of complete shrinkers with mixed ends \cite{CCS23}.

\bigskip

Besides rescaled ancient flows asymptotic to shrinkers, there are ancient flows converging to minimal surfaces as time goes negative. Marmor-Payne \cite{MP} constructed ancient graphical \textit{mean convex} flows over a certain class of unstable minimal hypersurfaces in $\mathbb{R}^{n+1}$  by taking the limit of a sequence of mean convex flows. In particular, their flows uniformly converge to the limit minimal hypersurfaces. On the other hand, the first author and Mantoulidis \cite{ChoiM} constructed an $I$-parameter family of ancient flows from closed minimal surfaces in manifolds by using the unstable eigenfunctions of the Jacobi operator of the limit minimal surface, where $I$ is the Morse index of the Jacobi operator. Also, given an unstable minimal hypersurface in $\mathbb{R}^{n+1}$ with finite total curvature, Han \cite{H} showed the existence of an $I$-parameter family of ancient graphical flows whose time slices converge to the limit minimal hypersurface with polynomial decay. Note that the polynomial convergence infers beautiful geometric properties of each time slice including finite total curvature. In this paper, given a two-sided complete unstable minimal hypersurface $\Sigma \subset\mathbb R^{n+1}$, we construct an $I$-family of ancient solutions with exponential convergence to the limit minimal hypersurfaces in space-time. To be specific, each flow converges to the graph of a solution $v:\Sigma \times (-\infty,T)\to \mathbb{R} $ to $v_t=\Delta_{\Sigma}v+|A_\Sigma|^2v$, where $\Sigma$ is the limit hypersurface. Moreover, the difference between the flow and the graph of $v$ is much smaller than the decay rate of $v$. This particularly yields the existence of ancient mean convex graphical flows with bounded geometric properties including finite total curvature and finite mass drop.  The following are our main theorems.  

\begin{theorem}\label{thm:main.I_parameter}
Let $\Sigma$ be a complete two-sided unstable embedded minimal hypersurface in $\mathbb{R}^{n+1}$ $(n\geq 2)$ with finite total curvature, and let $I$ denote the Morse index of $\Sigma$. Then, there exists an $I$-family of ancient mean curvature flows $M_t^{\mathbf{a}}$, with $\mathbf{a}\in B_{\varepsilon}(0)\subset \mathbb{R}^I$ for some $\varepsilon>0$, converging to $\Sigma$ exponentially fast in space-time as $|x|^2-t\to +\infty$. Moreover, each $M_t^{\mathbf{a}}$ is an entire graph over $\Sigma$ with finite total curvature and finite mass drop from $\Sigma$ for sufficiently negative $t$.
\end{theorem}

\begin{theorem}\label{thm:main mean}
Let $\Sigma$ be a complete two-sided unstable embedded minimal hypersurface in $\mathbb{R}^{n+1}$ $(n\geq 2)$ with finite total curvature. Then, for each open $U \subset \mathbb{R}^{n+1}$ with $\partial U =\Sigma$,
 there exists an ancient mean convex flow $M_t \subset U$ which is one of the $I$-parameter family $M_t^{\mathbf{a}}$ in Theorem \ref{thm:main.I_parameter}.
\end{theorem}

The proof and a more detailed statement of Theorem \ref{thm:main.I_parameter} and Theorem \ref{thm:main mean} are divided into three theorems: Theorem \ref{thm:exist}, Theorem \ref{thm:decay and finite mass}, and Theorem \ref{thm:speed}. The point in our construction is the  pointwise space-time exponential decay estimate of the profile function $u$ (see Theorem \ref{thm:decay and finite mass}) and precise error estimates. For example, the profile $u$ of a mean convex solution satisfies
\begin{equation}
     \|u-ae^{-\lambda_1 t}\phi_1\|_{C^{2,\alpha}(B_1(x)\times ( t-1, t))}\leq Ca^2e^{-2\lambda_1t}\phi_1(x)
\end{equation}
for some $a>0$, where $\phi_1 > 0$ and $\lambda_1<0$ are the first eigenfunction and eigenvalue of the Jaboci operator $\Delta_{\Sigma}+|A_\Sigma|^2$, respectively. These estimates appear in a quantitative way, which ensures us to give more precise characterizations of the ancient flows than the previous results. These quantitative estimates provide some geometric applications. For example, they imply that the ancient flows we constructed have finite total curvature and finite mass drop, and a one-parameter family of them is mean convex. 

We note that the MCF with finite total curvature
\begin{equation}
    \int_{M_t}|A|^n \, dvol_{M_t}
\end{equation}
has been studied by \cite{AY}, where they constructed ancient curve shortening flows with finite total curvature from multiple lines. Thus, our construction can be seen as a higher dimensional generalization of \cite{AY}.  Moreover,  Mramor \cite{M} recently proved a classification result for eternal mean convex flows from finite total curvature in $\mathbb{R}^3$. 

\bigskip

On the other hand, the mass drop
\begin{equation}
\begin{aligned}
     \int_{-\infty}^0\int_{M_t} H^2 \,dvol_{M_t}dt.
\end{aligned}
\end{equation}
is also an important quantity in the study of mean curvature flow, especially when we are considering the weak solutions, see \cite{I,MS,P} for instance. In particular, even if the minimal hypersurface in $\mathbb{R}^{n+1}$ has infinite mass, we can measure the mass drop of the ancient flows.
\bigskip

The paper is organised as follows. In Section \ref{sec prel}, we list some basic facts about complete unstable minimal hypersurfaces $\Sigma$ embedded in $\mathbb{R}^{n+1}$ with finite total curvature and we derive the equation for the mean curvature flow as a graph over $\Sigma$. In Section \ref{sec eigen}, we establish several estimates on the eigenfunctions of the Jacobi operator on $\Sigma$, which will be used to study the space-time asymptotics of our ancient flows. We employ the contraction mapping on each compact domain to construct the ancient flows with Dirichlet condition, and then take subsequential limits in Section \ref{sec exist}. This is the existence part of Theorem \ref{thm:main.I_parameter}. The precise exponential space-time decay of the ancient solutions and the mean convexity of one family of the flows are derived in Section \ref{sec asymp}. This completes the proof of the second part of  Theorem \ref{thm:main.I_parameter} and Theorem \ref{thm:main mean}.

\section{Preliminaries}\label{sec prel}
Consider a two-sided complete minimal hypersurface $\Si \subset \mathbb{R}^{n+1}$ with finite total curvature. We use $\nabla_i$ (and sometimes use subscript $i$ for short, for instance $u_i=\nabla_i u$) to denote the covariant derivatives on $\Si$, and use $D_i$ to denote the ordinary partial derivatives in $\mathbb{R}^n$. 

%subsection

%subsection
\subsection{Minimal hypersurfaces in $\mathbb{R}^{n+1}$ with finite total curvature}

Let $\Si$ be a two-sided complete minimal hypersurface in $\mathbb{R}^{n+1}$ ($n\geq 2$). We have the following well-known Jacobi operator on $\Si$:
\begin{align}\label{def:L}
\cL u := \Delta u + |A|^2 u,
\end{align}
 where $\Delta$ and $|A|$ denote the Laplacian-Beltrami operator and the square of the length of the second fundamental form of $\Si$ respectively. Associated to $\cL$ is the bilinear form $Q$ defined on $H^1(\Si)$
\begin{equation}
    Q(u,v)=\int_\Si(\nabla u\cdot\nabla v-|A|^2uv) \, dvol_\Si .
\end{equation}
 % $L$ is called the "stability" operator for area functional since
% \begin{align}
% \frac{d^2}{ds^2} \bigg|_{s=0} \mathrm{Area}(\text{graph}_M (su)) = \int_M u(-Lu),
% \end{align}
% for any smooth compactly supported function $u: M \rightarrow \mathbb{R}$. 

% We define the standard $L^2$ inner product for measurable functions $u, v: M \rightarrow \mathbb{R}$ as
% \begin{align}
% \langle u, v \rangle_2 := \int_M \langle u, v \rangle dvol_M.
% \end{align}
% This induces a norm $\|\cdot\|_2$ and a Hilbert space
% \begin{align}
% L^2(M) := \{u: M \rightarrow \mathbb{R} : \|u\|_2 < \infty\}.
% \end{align}
% on $M$. Similarly, we define the higher-order Sobolev spaces
% \begin{align}
% H^k(M) := \{ u: M \rightarrow \mathbb{R} : \|u\|_2 + \|\nabla u\|_2 + \cdots + \|\nabla^k u\|_2 < \infty\}.
% \end{align}
% where $\nabla^i u$ denotes the $i-th$ order covariant derivative of $u$. These are Hilbert spaces with respect to the inner product
% \begin{align}
% \langle u, v \rangle_k := \langle u, v \rangle_2 + \langle \nabla u, \nabla v \rangle_2 + \cdots + \langle \nabla^k u, \nabla^k v \rangle_2,
% \end{align}
% where the induced norm is denoted by $\|\cdot\|_k$. In these Hilbert spaces, $L$ is self-adjoint, i.e., 
% \begin{align}
% \langle Lu, v \rangle_2 = \langle u, Lv \rangle_2, \quad \text{for any} \quad u, v \in H^2(M).
% \end{align}

A function $u \in C_c^\infty(\Sigma)$, with $u \not \equiv 0$, is called an eigenfunction of $\mathcal{L}$ with eigenvalue $\lambda \in \mathbb{R}$ if
\begin{equation}
\mathcal{L} u = -\lambda u \qu \text{in }\Si.
\end{equation}

% $u\in C_c^\infty(\Si)$, $u\not \equiv 0$ is called an eigenfunction of $\cL$ with eigenvalue $\lambda\in\mathbb{R}$ if
% \begin{equation}
%  \cL u=-\lambda u.
% \end{equation}
% $M$ is called unstable if the quadratic form $Q(u,u)$ is not positive definite on $C^1_c(M)$, i.e. $L$ has eigenfunctions with negative eigenvalue. 

Let $O$ be a point on $\Si$ and $B_R(O)$ denote the intrinsic ball on $\Si$ of radius $R$ centered at $O$, with $\Si_R:= B_R(O)$. By standard theory, the operator $\cL$ on $\Si_R$ with Dirichlet boundary conditions maps $C^2_c(\Si_R)$ into $C_c^0(\Si_R)$. The spectrum of $\cL$ on $\Si_R$ is discrete, consisting of eigenvalues $\lambda_{R,1}<\lambda_{R,2}\leq\cdots$, where $\lambda_{R,i}$ go to infinity. Hence, each $\Sigma_R$ has finite Morse index $\text{index}(\Sigma_R)$. Moreover, it is monotone increasing in $R$. Therefore, we can define the index of $\Sigma$ as follows.

\begin{definition}
   The index of $\Si$ is defined as $\ind(\Si)=\lim_{R\to\infty} \ind(\Si_R)$. The hypersurface $\Si$ is called unstable if $\ind(\Si)>0$.
\end{definition}

We recall the definition of finite total curvature.

\begin{definition}
A hypersurface $\Si\subset \mathbb{R}^{n+1}$ is said to have finite total curvature if
\begin{equation}\label{finite total curvature}
\int_\Si|A|^n \, dvol_\Si<\infty.
\end{equation}
\end{definition}

 It is well-known (see \cite{F,Tysk fini} for example) that the finite total curvature assumption \eqref{finite total curvature} implies that $\Si$ has a finite index. Moreover, we have the following theorem (see \cite{F,Li}).
\begin{theorem}
    Let $\Si$ be an unstable two-sided minimal hypersurface in $\mathbb{R}^{n+1}$ with finite total curvature $\int_\Si|A|^n\,dvol_\Si$. Then the index $\ind(\Si)$ of $\Si$ is finite and there is a finite dimensional subspace $W$ of $L^2(\Si)$ having orthonormal basis $\phi_1,\cdots,\phi_k$ consisting of eigenfunctions with eigenvalues $\lambda_1,\cdots,\lambda_k$ respectively. Each $\lambda_i$ is negative and $Q(u,u)\geq 0$ for $u\in C_0^\infty(\Si)\cap W^\bot$. Moreover, $dim\,W=Ind(\Si)$.
\end{theorem}
We recall the definition and some properties for a minimal hypersurface which is regular at infinity.
\begin{definition}
[\cite{S,Li}] Suppose $n\geq 2$. A minimal hypersurface $\Si$ in $\mathbb R^{n+1}$ is regular at infinity if outside a compact set, each connected component $\Si_i$ of $\Si$ is the graph of a function $v$ over a hyperplane $P$, such that for $y\in P$,
\begin{equation}\label{fun asy 2}
    v(y)=a\log|y|+b+\frac{c_1y_1+c_2y_2}{|y|^2}+O(|y|^{-2})\qu \text{for some constants } a,b,c_1,c_2 \text{ if } n=2,
\end{equation}
and
 \begin{equation}\label{fun asy n}
     |y|^{n-2}|v(y)|+|y|^{n-1}|Dv(y)|+|y|^n|D^2v(y)|\leq C \qu \text {for some constant } C\text{ if }  n\geq 3.
 \end{equation}
Here, $f=O(|y|^{-2})$ implies $|f|\leq C |y|^{-2}$  in $\{|y|\geq R\}$ for some $C,R>0$. Each $\Si_i$ is called the end of $\Si$.
\end{definition}
 
\begin{proposition}
[\cite{S}] Suppose $\Si \subset\mathbb{R}^{3}$ is a complete immersed minimal surface with finite total curvature and each end $\Si_i$ of $\Si$ is embedded. Then $\Si$ is regular at infinity.
\end{proposition}
\begin{proposition}
[\cite{Tysk fini}] Suppose $n\geq 3$, $\Si\subset\mathbb{R}^{n+1}$ is a complete immersed minimal hypersurface with finite total curvature. Then $\Si$ is regular at infinity.
\end{proposition}
\begin{remark}\label{A bound g alpha}
    By the interior regularity theory for uniformly elliptic equations, a uniform $C^1$-bound for $v$ on $P$ implies uniform $C^k$-bounds for $v$ on $P$. For our purposes, we use the fact that if $\Si$ has finite total curvature, then $|A|,|\nabla|A||$ has linear decay on $\Si$, and the induced metric on $\Si$ tends to the Euclidean metric near infinity in the $C^{3,\alpha}$ sense for any $\alpha\in(0,1)$. In fact, if we use $y_1,\cdots,y_n$ to denote the coordinates on $P$, then the metric on $\Si_i$ is $g_{ij}=D_ivD_jv+\delta_{ij}=O(|y|^{-2(n-1)})+\delta_{ij}$, and the second fundamental form $h_{ij}= D_{ij}v(1+|Dv|^2)^{-\frac{1}{2}}=O(|y|^{-n})$, as $|y|\to\infty$ ($n\geq 2$) by \eqref{fun asy 2} and \eqref{fun asy n}.
\end{remark}

The study of complete minimal hypersurfaces has a long history. For more classical results in this direction, we refer the readers to \cite{JM,Osser 1, Na, CM}.  In the rest of the paper, $\Si$ will always denote a two-sided complete embedded minimal hypersurface in $\mathbb R^{n+1}$ with finite total curvature. Thus $\Si$ has a finite index and is regular at infinity.

\subsection{Graphical Mean Curvature Flow on Minimal Hypersurfaces}
Consider a two-sided complete embedded minimal hypersurface $\Si$. Let $M$ be a normal graph of $u:\Si\times \mathbb{R}$, namely
\begin{align}
F(x) = x + u(x) \nu(x)\in M.
\end{align}
where $\nu(x)$ is the unit normal vector of $\Sigma$ at $x \in \Sigma$. Here, we call $u$ the \textit{profile function} of $M$. We use the hat notation to indicate geometric quantities associated with $M$. For instance, we use $g_{ij}, h_{ij}$ to denote the induced metric and the second fundamental form of $\Si$, while $\hat{g}_{ij}$, $\hat h_{ij}$ to denote those of $M$ at $F(x)$. 

A direct computation shows
\begin{align}
\D_i F &= \D_i x+ u h_i^k\D_kx + u_i\nu, 
\end{align}
and the induced metric and outward unit normal vector are given by
\begin{align}
\hat{g}_{ij} &= \inn{\D_i F}{\D_jF}= g_{ij}+2uh_{ij}+u^2 h_i^kh_{kj}+ u_iu_j,
\\
\hat{\nu} &= V(\nu -\hat {g}^{ij}u_i\D_j F),
\end{align}
where $V= |\nu - \hat{g}^{ij} u_i\D_j F|^{-1}=(1-|\D u|_{\hat g}^2)^{-1/2}$. Moreover, since
\begin{align}
\D_{ij}^2  F &= (-h_{ij} - uh_i^kh_{kj} + u_{ij})\nu + (u_ih_j^k + u_jh_i^k + u(\D_j h_i^k))\D_kx,
\end{align}
the second fundamental form $\hat{h}_{ij} = \inn{-\D_{ij}^2  F}{  \hat{\nu}}$ of $M$ is 
\begin{equation}\label{hatA}
\begin{aligned}
\hat{h}_{ij} 
&= {V}^{-1}
(h_{ij} + uh_i^kh_{kj} - u_{ij}) +V(u_ih_j^k + u_jh_i^k + u(\D_j h_i^k))u_m\hat{g}^{mn}(g_{kn} + uh_{nk}).
\end{aligned}
\end{equation}
Note that $\inn{\nu}{ \hat{\nu}} = V(1 - |\D u|_{\hat g}^2)=V^{-1}$. Then the mean curvature is given by $\hat H = \hat g^{ij} \hat h_{ij}$, where $\hat g^{ij}$ is the inverse metric of $\hat g_{ij}$. 

The mean curvature flow implies the evolution equation of the profile $u$ that
\begin{align}\label{eq:MCF_u}
\pa_ t u = -V\hat H,
\end{align}%\marginpar{keep consistency. choose outward/inward to have the same sign in other paper}
with $V := \inn{\nu}{\hat\nu}^{-1} = (1 - |\nabla u|_{\hat g}^2)^{-\frac{1}{2}}$. The equation \eqref{eq:MCF_u} can be reformulated as
\begin{align}
\left(\pa_t - \cL\right) u = E(u) \quad \text{on } \Si \times \mathbb{R}_{-},
\end{align}
where $\cL$ is the operator specified in \eqref{def:L}, and $E(u)$ is 
\begin{align}\label{def:E}
E(u):=-V\hat H - \cL u 
=&\,(\hat g^{ij}-g^{ij})(u_{ij}-h^k_ih_{kj}u)-h_{ij}(\hat g^{ij}-g^{ij}+2 g^{im}g^{nj}h_{mn}u)
\\
&\qu - \frac{u_m\hat{g}^{ij}\hat{g}^{mn}}{1 - u_iu_j\hat{g}^{ij}} (u_ih_j^k + u_jh_i^k + u\D_jh_i^k)( g_{kn} + uh_{nk}).
\end{align}
Assuming that $\|u\|_{C^2(\Sigma)}\ll 1$, the asymptotic expansion yields
\begin{align}
\hat g^{ij} = g^{ij} - 2h^{ij}u + O(\|u\|_{C^1(\Sigma)}^2),
\end{align}
and therefore 
\begin{align}\label{eq:EO}
E(u) = O(\|u\|_{C^2(\Sigma)}^2).
\end{align}

\subsection{Notations}
In this paper, we denote by $\rho(x_1, x_2)$ the intrinsic distance between points $x_1$ and $x_2$ in $\Si$. Then, we denote by  $B_R(x):=\{x'\in \Si:\rho(x,x')<R\}$ the intrinsic ball in $\Si$ centered at $x\in \Si$ with radius $R$, and for a fixed point $O \in  \Si$ we define
\begin{align}
    \Si_R:=B_R(O).
\end{align}

\bigskip

Now, we let $I_R$ denote the Morse index of $\Sigma_R$, and denote by $\{\phi_{R,i}\}_{i=1}^{I_R}$ the orthonormal Dirichlet eigenfunctions of $\mathcal{L}$ on $\Sigma_R$ with the corresponding eigenvalues $\lambda_{R,1}<\lambda_{2,R}\leq \cdots \leq \lambda_{R,I_R}<0$. Then, in the same manner we denote by $\phi_i$ and $\lambda_i$ the orthornomal eigenfunctions and eigenvalues with $\lambda_1\leq \lambda_2\leq \cdots \leq \lambda_I<0$, where $I=\text{index}(\Sigma)$. 

Moreover, we use $L$ denote the number of distinct negative eigenvalues of $\cL$ on $\Sigma$, and denote by $\la_1=\la'_1< \la'_2< \cdots< \la'_{L}$ the corresponding distinct negative eigenvalues. We use $\lambda_R = \lambda_{R,1}$, $\phi_R = \phi_{R,1}$; $\lambda = \lambda_1$, $\phi = \phi_1$ to be the simplified notations for the first eigenvalue and eigenfunction on $\Si_R$ and $\Si$, respectively.
\bigskip

Finally, we note that, $I_R\nearrow I$, $\lambda_{R,i}\searrow \lambda_i$, and $\phi_{R,i}\to\phi_i$ as $R\nearrow\infty$. Thus, in the following, we will choose $R$ sufficiently large, such that $I_R=I$.

%\end{itemize}

%section
\section{Eigenfunctions of $\cL$}\label{sec eigen}
%\marginpar{some property holds for any sol, others holds when it is eigenfunction}
%Let $\Si$ be a two-sided unstable complete embedded minimal hypersurface in $\mathbb{R}^{n+1}$ with finite total curvature. 
Recall that $\phi$ is the first eigenfunction of $\cL=\De+|A|^2$ on $\Si$ and $\la$ the corresponding first eigenvalue. That is, $\phi>0$ in $\Si$ and $\la<0$ satisfy
\begin{align}\label{eq phi}
\De \phi + |A|^2 \phi +\la \phi = 0 \qu \text{in } \Si.
\end{align}
\subsection{Harnack and H\"older estimate}

%paragraph
We first establish a Harnack inequality for $\phi$.
%theorem
\begin{theorem}[Harnack inequality]\label{harnack}
Suppose $\phi>0$ denotes the first eigenfunction of the operator $\cL$ on $\Si$. Then, it holds that
\begin{equation}
|\D \log \phi| \le C \qu \text{on } \Si,
\end{equation}
where $C=C(n,|\lambda|, \sup_\Si |A|, \sup_\Si |\nabla |A|^2|) > 0$ is a constant.
\end{theorem} %\marginpar{check constant dependence}

\begin{proof}
In the following proof, $C$ denotes a constant depending only on $n,|\lambda|, \sup_\Si |A|, \sup_\Si |\nabla |A|^2|$, which may change from line to line. Let $v = \log \phi$ and $w = |\nabla v|^2$. From equation (\ref{eq phi}),
\begin{equation}\label{eq vlog}
\Delta v = -w - |A|^2 - \lambda.
\end{equation}
Note that 
\begin{equation}
\D_k\D_k\D_iv = \D_k\D_i\D_k v = \D_i\D_k\D_k v + R_{ip}\D^pv,    
\end{equation}
and the Gauss equation gives 
\begin{equation}
R_{ik} = Hh_{ik} - h_{im}h^{m}_k = - h_{im}h^{m}_k=-h_i^jh_{jk}\geq-|A|^2\delta_{ik},    
\end{equation}
where $R_{ik}$ is the Ricci curvature. Thus
\begin{align}
\D^iv\De(\D_iv)=\D^iv \D_i(\De v)+R_{ip}\D^iv\D^p v \ge -\D^iv \D_i(w+|A|^2)-|A|^2 w.
\end{align}
On the other hand, by the Cauchy--Schwarz inequality and \eqref{eq vlog}, we obtain 
\begin{equation}
    2\D_k\D_i v \D^k\D^i v \ge \tfrac{2}{n}(\De v)^2 =\tfrac{2}{n} (w+|A|^2+\la)^2.
\end{equation}
Hence,
\begin{align}\label{eq:w}
\De w = 2\D^i v \De (\D_iv)+2 \D_k\D_i v\D^k\D^iv\ge 
-2\D v\cdot \D w -2w^{\fr{1}{2}}|\D |A|^2| -2|A|^2 w+\tfrac{2}{n}(w+|A|^2+\la)^2.
\end{align}

\bigskip

Define a cutoff function as $\eta = ((2R)^2 - \rho^2)_+$, where $R \geq 1$ is a constant, and $\rho(x)$ denotes the intrinsic distance from a fixed point $O \in \Si$. Then on the set $\{\eta>0\}$,
\begin{equation}\label{grd eta}
    \D_i\eta = -2\rho\rho_i \quad \text{and} \quad \De \eta = -2\rho \De \rho -2.
\end{equation}
Since $\Si$ has finite total curvature, 
\begin{equation}
R_{ij}\geq -|A|^2\delta_{ij} \geq -C\delta_{ij},
\end{equation}
by Remark \ref{A bound g alpha}. The Laplacian comparison theorem (Corollary 1.2 in Chapter 1 of \cite{SY}) yields $\rho    \Delta\rho \leq (1 + C\rho)$, which implies
\begin{equation}\label{del eta}
    \Delta\eta^2 = 2\eta \De \eta +2|\D\eta|^2\geq -C\eta (1+\rho)+ 2|\D \eta|^2.
\end{equation}

We now define $F = \eta^2 w$. Then $F$ attains its maximum at some point $x_0\in \Si$ by definition of $F$.

If $F(x_0)=0$, then $F(x)\leq F(x_0)=0$ for $x\in Si_R$, which implies that $w\equiv 0$ on $\Si_R$. We are done. 

If $F(x_0)>0$, then $x_0\in \bar \Si_{2R}\setminus\partial \Si_{2R}$ and $w(x_0)>0$. By the first derivative test, we have at $x_0$
\begin{equation}\label{firs de}
    \D w = -2w\eta^{-1}\D \eta \qu \text{and so}\qu -2\D v \cdot \D w\ge -8\eta^{-1}  \rho w^{\fr{3}{2}}.
\end{equation}
By maximum principle and using \eqref{eq:w}, \eqref{grd eta}, \eqref{del eta}, \eqref{firs de},
\begin{equation}
    \begin{aligned}
        0&\ge\Delta F = (\Delta \eta^2)w + \eta^2\big(\De w\big) 
        - 8|\D \eta|^2w\\
        &\ge -C\eta (1+\rho)w +\eta^2(-2\D v\cdot \D w -2w^{\fr{1}{2}}|\D |A|^2| -2|A|^2 w+\tfrac{2}{n}(w+|A|^2+\la)^2)-6|\D\eta|^2w\\
        &\ge \tfrac{2}{n}\eta^2 w^2-8\eta \rho w^{\frac{3}{2}}+\eta^2(((\tfrac{4}{n}-2)|A|^2+\tfrac{4}{n}\lambda)w-2w^{\fr{1}{2}}|\D |A|^2|)-C\eta (1+\rho)w-24 \rho^2 w.
    \end{aligned}
\end{equation}
Multiplying $\eta^2$ and noting $\rho(x_0)< 2R$ with $R\ge1$, we get
\begin{align}\label{max 1}
0\ge F^2-CR^2 F^{3/2}-C(1+|A|^2+|\la|)R^4 F  -C|\D |A|^2|R^6 F^{\fr{1}{2}},
\end{align}
and therefore, $F\le CR^4$ at $x_0$. Since $\eta^2(x)\ge 3R^4$ for $x\in \Si_R$, we conclude that
\begin{align}
w(x)=F(x_0)/\eta^2(x)\le C \qu \text{for } x\in \Si_R,
\end{align}
where $C$ is independent of $R$. Then the desired result follows by taking $R\ra \infty$.
\end{proof}

\subsection{Eigenfunctions on $\Si_R$}
In this subsection, we derive the estimates for eigenfunctions of $\cL$ on $\Si_R$ with Dirichlet condition in terms of $\phi$. Recall that $I_R$ is the index of $\cL$ with Dirichlet condition on $\Si_R$ and $\phi_{R,i}\in C^\infty_c(\Si_R)$ the $i$-th eigenfunction of $\cL$ on $\Si_R$ with eigenvalue $\lambda_{R,i}<0$. That is, $\phi_{R,i}$ satisfies the equation
 \begin{equation}\label{phiri eq}
  \left\{
    \begin{aligned}
        \cL\phi_{R,i}&=-\lambda_{R,i}\phi_{R,i}\qu\text{in }\Si_R,\\
        \phi_{R,i}&=0 \qu\text{on }\partial \Si_R.
    \end{aligned}
    \right.
    \end{equation}
% These estimates will be used in the estimates of  
%lemma
\begin{lemma}\label{quo boun}
Let $\phi$ be the first eigenfunction of $\mathcal{L}$ on $\Sigma$, and let $\phi_{R,i}$ (for $i \leq I_R$) be the $i$-th eigenfunction of $\mathcal{L}$ on $\Sigma_R$ with eigenvalue $\lambda_{R,i} < 0$. Given any $\varepsilon \in (0, \frac{1}{2})$, there exist constants $C = C(\Sigma, \varepsilon) > 0$ and $R_0 = R_0(\varepsilon) > 0$ such that if $R > R_0$, then
\begin{equation}\label{quo bou eig}
|\phi_{R,i}| \leq C \phi^{(1-\varepsilon)\frac{\lambda_{R,i}}{\lambda}}.
\end{equation}
\end{lemma}

\begin{proof}
Note $\Delta\phi + (|A|^2 + \lambda)\phi = 0$, and let $\hat\psi=\phi^\gamma$. Then we have
\begin{align}
    \Delta(\tfrac{\phi_{R,i}}{\hat\psi})=&\tfrac{\Delta\phi_{R,i}}{\hat\psi}-\tfrac{\phi_{R,i}\Delta\hat\psi}{\hat\psi^{2}}-2\tfrac{\nabla\hat\psi}{\hat\psi}\cdot\nabla(\tfrac{\phi_{R,i}}{\hat\psi})\\
    =&\tfrac{-(|A|^2+\lambda_{R,i})\phi_{R,i}}{\hat\psi}+\tfrac{\gamma(|A|^2+\lambda)\phi_{R,i}}{\hat\psi}-\tfrac{\gamma-1}{\gamma}\tfrac{\phi_{R,i}}{\hat\psi}\tfrac{|\nabla\hat\psi|^2}{\hat\psi^2}-2\tfrac{\nabla\hat\psi}{\hat\psi}\cdot\nabla(\tfrac{\phi_{R,i}}{\hat\psi})\\
    =&((\gamma-1)|A|^2+\gamma\lambda-\lambda_{R,i}+(1-\gamma)\gamma\tfrac{|\nabla\phi|^2}{\phi^2})\tfrac{\phi_{R,i}}{\hat\psi}-2\gamma\tfrac{\nabla\phi}{\phi}\cdot\nabla(\tfrac{\phi_{R,i}}{\hat\psi}).
\end{align}
Fix $\varepsilon \in (0, \frac{1}{2})$. Define $\gamma := (1-\varepsilon)\frac{\lambda_{R,i}}{\lambda}<1$. Then we have
\begin{equation}\label{eq:phiri_phi}
    \Delta\left(\tfrac{\phi_{R,i}}{\hat{\psi}}\right) = \left((\gamma-1)|A|^2 - \varepsilon \lambda_{R,i} + (1-\gamma)\gamma \tfrac{|\nabla \phi|^2}{\phi^2}\right)\tfrac{\phi_{R,i}}{\hat{\psi}} - 2\gamma \tfrac{\nabla \phi}{\phi} \cdot \nabla \left(\tfrac{\phi_{R,i}}{\hat{\psi}}\right).
\end{equation}
Since $|A|^2 \leq \frac{C}{1 + |x|^2}$ by Remark \ref{A bound g alpha} and $\lambda_{R,i} \searrow \lambda_i < 0$ as $R \nearrow \infty$, there exists a sufficiently large $R_0(\varepsilon) > 0$ such that
\begin{equation}
    (\gamma-1)|A|^2 - \varepsilon \lambda_{R,i} + (1-\gamma)\gamma \tfrac{|\nabla \phi|^2}{\phi^2} < 0 \quad \text{for } x \in \Sigma \setminus \Sigma_{R_0}.
\end{equation}

On the other hand, we note that $|\phi_{R,i}|=0\leq \hat\psi$ on $\partial \Si_R$. Additionally, as $\phi_{R,i} \to \phi$ in $C^\infty_{loc}(\Si)$ and $\phi>0$ is uniformly bounded away from zero on $\Si_{R_0}$, we obtain a uniform bound for $\hat\psi^{-1}\phi_{R,i}$ on $\bar{\Si}_{R_0}$ for large enough $R \geq R_0$. 
Applying the maximum principle to $\hat\psi^{-1}\phi_{R,i}^+$ and $\hat\psi^{-1}\phi_{R,i}^-$ on $\Si_R \setminus \Si_{R_0}$, we conclude that %$\phi_R \leq C\phi^{1-\varepsilon}$ for $R > R_0$, 
\begin{equation}
    \phi_{R,i}^+,\phi_{R,i}^-\leq C\hat\psi=C\phi^{(1-\varepsilon)\tfrac{\lambda_{R,i}}{\lambda}},\quad R\geq R_0(\varepsilon).
\end{equation}
where $C$ depends on $\Si$ and $\varepsilon$. Here, $f^+=\max\{f,0\},f^-\max\{-f,0\}$ denote the positive and negative parts of $f$ respectively. This establishes the lemma.
\end{proof}

\begin{theorem}\label{2 alpha point quo}
Let $\phi$ be the first eigenfunctions of $\cL$ on $\Si$, and let $\phi_{R,i}$ ($i\leq I_R$) be the $i$-th eigenfunction of $\cL$ on $\Si_R$ with eigenvalue $\lambda_{R,i}<0$. For every $\varepsilon\in(0,\tfrac{1}{2})$, there exists $R_2=R_2(\varepsilon,\Si)$, $C=C(\varepsilon)$ large, such that for any $R\geq R_2$,
    \begin{equation}\label{2 alpha quo boun}
        \|\phi_{R,i}\|_{C^{2,\alpha}(B_1(x))} \leq C\phi^{(1-\varepsilon)\tfrac{\lambda_{R,i}}{\lambda}}(x)\quad x\in \Si.
    \end{equation}
\end{theorem}

%proof
\begin{proof}
By the Harnack inequality (\Cref{harnack}), 
we have 
\begin{equation}
\sup_{B_2(x)}\phi\leq C\inf_{B_2(x)}\phi\le C \phi(x),\quad \text{for }x\in \Si,
\end{equation}
where $C$ is a constant independent of $x$. Applying the Schauder estimate to $\phi_{R,i}$ and \Cref{quo boun}, for $R\ge R_1(\e)$, we obtain for any $x\in \Si$
\begin{align}
\|\phi_{R,i}\|_{C^{2,\alpha}(B_1(x))}\leq&\, C\|\phi_{R,i}\|_{L^\infty(B_2(x))}
\leq C\sup_{B_2(x)}\phi^{(1-\varepsilon)\tfrac{\lambda_{R,i}}{\lambda}}
\leq C\phi(x)^{(1-\varepsilon)\tfrac{\lambda_{R,i}}{\lambda}},\quad x\in \Si,
\end{align}
which completes the proof.
\end{proof}

\begin{remark}
Using a similar argument, one can show the following bound for $\phi$: 
\begin{align}\label{eq:phi2a}
\|\phi\|_{C^{2,\alpha}(B_1(x))}\le C\phi(x),\quad x\in \Si,
\end{align}
with $C$ independent of $x$.
\end{remark}

\begin{lemma}\label{lem hat error quo}
If $E(x, t)$ is a function defined on $\Si\times\mathbb R_-$ satisfying 
\begin{equation}
    \|E\|_{C^\alpha(B_1(x)\times( t-1, t))}\leq Ce^{-2\lambda_{R,i} t}\phi^{2(1-\varepsilon)\frac{\lambda_{R,i}}{\lambda}}(x),
\end{equation}
with $C$ independent of $x$, then
\begin{equation}\label{hat error quo alpha l}
    \| E\phi^{-2(1-\varepsilon)\frac{\lambda_{R,i}}{\lambda}}\|_{C^\alpha(B_1(x)\times( t-1, t))}\leq Ce^{-2\lambda_{R,i} t}.
\end{equation}
\end{lemma}
\begin{proof}
    We note that for any functions $f,g$ with $g>0$, we have
    \begin{align}
        \nabla (\tfrac{f}{g})=&\tfrac{\nabla f}{g}-\tfrac{f}{g}\tfrac{\nabla g}{g}\\
        \nabla^2(\tfrac{f}{g})=&\tfrac{\nabla^2 f}{g}-2\tfrac{\nabla f}{g}\otimes\tfrac{\nabla g}{g}-\tfrac{f}{g}\tfrac{\nabla^2g}{g}+2\tfrac{f}{g}\tfrac{\nabla g\otimes\nabla g}{g^2}\\
        %(\frac{f}{g})
        (\tfrac{f}{g})(x_1)- (\tfrac{f}{g})(x_2)=&\tfrac{f(x_1)-f(x_2)}{g(x_1)}+\tfrac{f(x_2)}{g(x_1)}\tfrac{g(x_2)-g(x_1)}{g(x_1)}\\
        (fg)(x_1)-(fg)(x_2)=&(f(x_1)-f(x_2))g(x_1)+f(x_2)(g(x_1)-g(x_2)).
    \end{align}
    By applying these identities with suitable $f,g$ and \eqref{eq:phi2a}, we have \eqref{hat error quo alpha l}.    
\end{proof}

%section
\section{Existence of ancient Solutions}\label{sec exist}

%subsection
\subsection{Ancient Solutions on $\Si_R$} 
Our aim is to identify the existence of ancient solutions for equation \eqref{eq:MCF_u}. To achieve this, we initially focus on finding such solutions on the domain $\Si_R$ with the Dirichlet boundary condition. Specifically, we seek solutions for:
\begin{equation}\label{eq:MR}
\left\{
\begin{aligned}
     \partial_ t u_R = \mathcal{L}u_R + E(u_R) \quad  \text{in }\Si_R\times\mathbb{R}_-,\\
     u_R=0 \qu\text{on } \pa\Si_R\times\mathbb{R}_-.
\end{aligned}
\right.
\end{equation}
%For the purposes of this subsection, $R$ will be fixed as a large number so that $\ind(\Si_R) = \ind(\Si)$, and the notation will be simplified by omitting the explicit reference to $R$ in the solution and operator.

Following the argument in \cite{CS1}, we utilize the contraction mapping technique. This method necessitates establishing the following estimates for the error $E$.

%lemma
\begin{lemma}\label{lem:E}
There exists a constant $\e=\varepsilon(\si, \alpha) > 0$, independent of $R$, such that if $\|u\|_{C^2} \leq \varepsilon$, then the following inequalities hold for some constant $C=C(\Si,n,\al)$:
\begin{align}
        |E(u)| &\leq C \|u\|^2_{C^2},\label{c0}
        \\
        \|E(u)( t)\|_{C^\alpha} &\leq C \|u\|_{C^{1,\alpha}} \|u\|_{C^{2,\alpha}}.\label{holder}
\end{align}
Furthermore, if both $u$ and $v$ satisfy the condition $\|u\|_{C^2} + \|v\|_{C^2} \leq \varepsilon$, then
    \begin{equation}\label{diff holder}
        \|E(u)( t) - E(v)( t)\|_{C^\alpha} \leq C \|(u - v)( t)\|_{C^{2,\alpha}} (\|u( t)\|_{C^{2, \alpha}} + \|v( t)\|_{C^{2, \alpha}}).
    \end{equation}
\end{lemma}

%proof
\begin{proof}
We note that $E$ depends on $u,\nabla u,\nabla^2u$ from \eqref{def:E}, so we treat $E=E(z,p,r)$ as a function of $(z,p,r)\in\mathbb\mathbb R\times\mathbb R^n\times\mathbb R^{n\times n}$, and use $\tilde{\nabla}$ to denote the derivative of $E$ with respect to $(z,p,r)$. % write $E(u)=E(u,\nabla u,\nabla^2u)$ by abuse of notation. 
To simplify the notation, let $q(u)$ represent the triplet $(u, \nabla u, \D^2 u)$. Define $f(s) = E(q(su)) = E(su, s\nabla u, s\D^2 u)$ for $s \in [0,1]$. By the definition of $E$, \eqref{def:E}, it follows that $f(0) = f'(0) = 0$. Then, we have
\begin{align}
E(u, \nabla u, \D^2u) =   \int_0^1 (1-s)f''(s) ds= \int_0^1 (1-s) \tilde\D^2_{q(u),q(u)} E ds,
\end{align}
where the Hessian of $E$ in the integral is evaluated at $q(su) = (su, s\nabla u, s\D^2u)$. The first conclusion follows directly since the Hessian of $E$ is bounded and we have
\begin{align}
|\tilde\D^2_{q(u), q(u)} E| \leq |\tilde\D^2 E|_{L^\infty} |q(u)|_{L^\infty}^2 \leq C \|u\|_{C^2}^2.
\end{align}
    
\bigskip

For the second inequality, utilizing the relation $|\int_0^1 f(x,\cdot)dx|_{C^\alpha} \leq \int_0^1 |f(x,\cdot)|_{C^{\alpha}}dx$ with $|fg|_{C^\alpha} \leq |f|_{\mathcal{C}^0}|g|_{C^\alpha} + |f|_{C^\alpha}|g|_{C^0}$ and $|f \circ g|_{C^\alpha} \leq |\nabla f|_{C^0}|g|_{C^\alpha}$, we obtain for any $ t \in (-\infty, 0]$,
\begin{align}
    |E( t)|_{C^\alpha} \leq C(n)\|u( t)\|_{C^{2,\alpha}}\|u( t)\|_{C^{1,\al}},
\end{align}
where we used $\frac{\partial^2 E}{\partial u_{ij}\partial u_{kl}}=0$.

\bigskip

Let $t_1, t_2$ be in the interval $[ t-1,  t]$ and $x_1, x_2$ belong to $\Si_R$. For $i = 1, 2$, define $f_i(s) = sq(u)(x_i) + (1-s)q(v)(x_i)$, which represents a linear interpolation between $q(u)(x_i)$ and $q(v)(x_i)$. Then we have
\begin{align}
I =& E(f_1(1)) - E(f_1(0)) - E(f_2(1)) + E(f_2(0)) \\
=& \int_0^1 \tilde\nabla E (f_1(s)) \cdot \left(\tfrac{d f_1(s)}{ds} - \frac{d f_2(s)}{ds}\right)  ds - \int_0^1 \left(\tilde\nabla E (f_2(s)) - \tilde\nabla E (f_1(s))\right) \cdot \tfrac{d f_2(s)}{ds}  ds.
\end{align}
Since $\frac{d f_i(s)}{ds}=q(u-v)(x_i)$, for the first integral, we have
\begin{align}
\int_0^1 \tilde\nabla E (f_1(s)) \cdot \left(q(u-v)(x_1)-q(u-v)(x_2)\right)  ds
\le C|\tilde\D E|\|u-v\|_{C^{2,\alpha}}(|x_1-x_2|^\alpha+|t_1-t_2|^{\alpha/2}),
\end{align}
and for the second integral, we obtain
\begin{align}
&\int_0^1 \left(\tilde\nabla E (f_2(s)) - \tilde\nabla E (f_1(s))\right) \cdot \tfrac{d f_2(s)}{ds}  ds
\le C|\tilde\D^2E|(f_2(s)-f_1(s))\|u-v\|_{C^2}
\\
&\le C|\tilde\D^2E|(\|u\|_{C^{2,\alpha}}+\|v\|_{C^{2,\alpha}})\|u-v\|_{C^2}(|x_1-x_2|^\alpha+|t_1-t_2|^{\alpha/2}).
\end{align}
Combining these two inequalities, we have the desired result.
\end{proof}

To prove the existence of solutions \eqref{eq:MR}, we need the following definitions. 
For $\alpha\in (0,1)$ and a function $f:\Si\times \mathbb{R}_-\to\mathbb{R}$, we define 
\begin{equation}
    \begin{aligned}
    \|f\|_{l,\alpha;R, t}:&=\sum_{i+2j\leq l}\sup_{\Si_R\times( t-1, t)}|\partial_x^i\partial_t^j f|+\sum_{i+2j=l}[\partial_x^i\partial_t^jf]_{\al;R, t},\\
    \end{aligned}
\end{equation}
where
\begin{align}
[f]_{\alpha;R, t} &:= \sup_{\substack{(x_i,t_i) \in \Si_R \times ( t-1, t) \\ (x_1,t_1) \neq (x_2,t_2)}} \left\{\frac{|f(x_1,t_1) - f(x_2,t_2)|}{d(x_1,x_2)^\alpha + |t_1-t_2|^{\frac{\alpha}{2}}} \right\}.
\end{align}
We also define $X_R^\de$ to be the set of functions $f:\Si\ti \R_-\ra \R$ so that
\begin{align}
\|f\|_{X^{\delta}_R}:=\|f\|_{{2,\alpha,\delta;R}}:&=\sup_{ t\leq 0}\{e^{-\delta t}\|f\|_{2,\alpha;R, t}\} < \infty.
\end{align}
Similarly, we define $L^{2,\de}_R(\Si_R\ti \R_-)$ to be the set of functions $f:\Si\ti \R_-\ra \R$ so that
\begin{align}
\|f\|_{L^{2,\delta}_R}:&=\sup_{ t\leq 0}\{e^{-\delta t}\|f(\cdot, t)\|_{L^2(\Si_R)}\}.
\end{align}

We fix an $L^2$ orthonormal sequences of eigenfunctions $\phi_{R,j}$ of $\cL$ on $\Si_R$ with the Dirichlet condition such that $\mathcal{L}\phi_{R,j}=-\lambda_{R,j}\phi_{R,j}$ and $(\phi_{R,i},\phi_{R,j})_{L^2}=\de_{ij}$, where $\phi_{R,j}\in H^1_0(\Si_R)$. Then $\{\phi_{R,j}\}_{j=1}^\infty$ forms an orthonormal basis of $L^2(\Si_R)$. Define $v_j=(v,\phi_{R,j})$ and $P_{R,j}v=(v,\phi_{R,j})\phi_{R,j}$, and 
\begin{equation}
    P_{R,-}=\sum_{j=1}^IP_{j},\quad P_{R,0}=\sum_{j=I+1}^{I+K_R}P_{R,j}, \quad P_{R,+}=Id-(P_{R,-}+P_{R,0}),
\end{equation}
where $I=\ind(\mathcal{L})$, $K_R=\operatorname{dim\, ker}(\mathcal{L})$ on $\Si_R$ (Recall that we choose $R_0$ large such that $I_R=I$ for $R\geq R_0$).
We choose $\delta>0$, $\delta\neq-\lambda_{R,j}$ for any $j$. Denote by $J=\{j:\lambda_{R,j}<-\delta\}\subset\{1,\cdots, I\}$. 

When $R=\infty$, i.e., $\Si_R=\Si$, we omit $R$ in the above norm. For instance, we write 
\begin{equation}
\begin{aligned}
&\|f\|_{l,\alpha, t}:=\sum_{i+2j\leq l}\sup_{\Si\times( t-1, t)}|\partial_x^i\partial_t^j f|+\sum_{i+2j=l}[\partial_x^i\partial_t^jf]_{\al; t},\\
   & X^\delta=\{f\in C^{2,\alpha}(\Si\times\mathbb R_-):\|f\|_{X^{\delta}}:=\|f\|_{{2,\alpha,\delta}}:=\sup_{ t\leq 0}\{e^{-\delta t}\|f\|_{2,\alpha; t}\}<\infty\}.
\end{aligned}
\end{equation}
We also note that for $f\equiv 0$ on $(\Si\setminus \Si_R)\times\mathbb R_-$, the norms with $R$ and without $R$ are the same, i.e., 
\begin{equation}
    \|f\|_{l,\alpha,\delta;R}=\|f\|_{l,\alpha,\delta}\qu \text{ if }f\equiv 0 \text{ on }(\Si\setminus \Si_R)\times\mathbb R_-.
\end{equation}

\begin{lemma}\label{app solu}
    Fix any $0<\delta\notin\{-\lambda_{R,j}\}_{j=1}^\infty$. If $\|f\|_{L_R^{2,\delta}}<\infty$, then the equation
    \begin{equation}
  \left\{
    \begin{aligned}
         \partial_ t u-\mathcal{L}u&=f(x, t) \qu\text{in }\Si_R\times\mathbb{R}_-,\\
        u&=0 \qu\text{on } \pa(\Si_R\times\mathbb{R}_-)
    \end{aligned}
    \right.
    \end{equation}
    has a unique solution $u$ satisfying $\|u\|_{L^{2,\delta}_R}<\infty$ and $P_{R,j}(u(\cdot,0))=0$ for $j\in J$. Furthermore, there exists a constant $C=C(\Si,\alpha,\delta)$, independent of $R\geq R_0$, such that $\|u\|_{L_R^{2,\delta}}\leq C\|f\|_{L_R^{2,\delta}}$ and $\|u\|_{X^{\de}_R}\leq C\|f\|_{0,\alpha,\delta;R}$.
\end{lemma}

% \begin{proof}
%     The proof is the same as that of Lemma 3.1 of \cite{CS1} together with a covering argument.
% \end{proof}
\begin{proof}
   It suffices to solve
   \begin{equation}
       \partial_ t u_{R,i}+\lambda_{R,i}u_{R,i}=f_{R,i}\quad \text{on }\Si_R\times\mathbb{R}_-,
   \end{equation}
   where $u_{R,i}=(u,\phi_{R,i})$, and $f_{R,i}=(f,\phi_{R,i})$. Define
   \begin{equation}
       \begin{aligned}
           u_{R,j}( t):&=-\int_ t^0e^{\lambda_{R,j}(s- t)}f_{R,j}(s)ds,\quad j\in J,\\
           u_{R,j}( t):&=\int_{-\infty}^ t e^{\lambda_{R,j}(s- t)}f_{R,j}(s)ds,\quad j\in J^{c}=\mathbb{Z}_+\setminus J.
       \end{aligned}
   \end{equation}
   Note that the integral on the RHS is well-defined because $|f_{R,j}(s)|=|(f,\phi_{R,j})|\leq \|f(\cdot,s)\|_{L^2(\Si_R)}\leq \|f\|_{L^{2,\delta}_R}e^{\delta s}$. Define $u_R(\cdot, t)=\sum_{j=1}^\infty u_{R,j}( t)\phi_{R,j}(\cdot)$. By definition of $u_{R,j}( t)$ for $j\in J$, we have $P_{R,j}(u_R(\cdot,0))=0$ for $j\in J$. 
Choose $\delta'$ and $\delta''$ satisfying $\max_{j\in J}\{\lambda_{R,j}\}<-\delta'<-\delta<-\delta''<\min_{j\in J^{c}}\{\lambda_{R,j}\}$. Note that for $j\in J$,
\begin{align}
   u^{2}_{R,j}( t)\leq \int_ t^0 e^{2(\lambda^R_j+\delta')(s- t)}ds\int_{ t}^0e^{-2\delta'(s- t)}|f_{R,j}(s)|^2ds\leq C\int_ t^0e^{-2\delta'(s- t)}|f_{R,j}(s)|^2ds
\end{align}
and for $j\in J^{c}$,
\begin{align}
   u_{R,j}^{2}( t)\leq\int_{-\infty}^ t e^{2(\lambda_{R,j}+\delta'')(s- t)ds}\int_{-\infty}^ t e^{-2\delta''(s- t)}|f_{R,j}(s)|^2ds\leq C\int_{-\infty}^ t e^{-2\delta''(s- t)}|f_{R,j}(s)|^2ds.
\end{align}
Using $\|f(s)\|_{L^2(\Si_R)}\leq \|f\|_{L_R^{2,\delta}}e^{\delta s}$, we then have
\begin{align}
   \|u_R(\cdot, t)\|^2=\sum_j u^{2}_{R,j}( t) %& C\sum_{j}\int_{ t}^0 e^{-2\delta'(s- t)}|f_{R,j}(s)|^2ds +\int_{-\infty}^ t e^{-2\delta''(s- t)}|f_{R,j}(s)|^2ds\\=
  % &C(\int_ t^0e^{-2\delta'(s- t)}\sum_{j\in J}|f_{R,j}(s)|^2 ds+\int_{-\infty}^ t e^{-2\delta''(s- t)}\sum_{j\in J^{c}}|f_{R,j}(s)|^2ds)\\
   \leq& C\left(\int_ t^0e^{-2\delta'(s- t)}\|f(s)\|_{L^2(\Si_R)}^{2} ds+\int_{-\infty}^ t e^{-2\delta''(s- t)}\|f(s)\|_{L^2(\Si_R)}^{2}ds\right)\\
   \leq & C\left(\int_ t^0e^{-2\delta'(s- t)}\|f\|_{L_R^{2,\delta}}^{2}e^{2\delta s} ds+\int_{-\infty}^ t e^{-2\delta''(s- t)}\|f\|_{L_R^{2,\delta}}^{2}e^{2\delta s}ds\right)\\
   = &C\|f\|^{2}_{L_R^{2,\delta}}e^{2\delta t}\left(\int_{ t}^0 e^{2(\delta-\delta')(s- t)}ds+\int_{\infty}^ t e^{2(\delta-\delta'')(s- t)}ds\right)\\
   \leq &C\|f\|^{2}_{L_R^{2,\delta}}e^{2\delta t}.
\end{align}
Hence, $\|u_R\|_{L_R^{2,\delta}}\leq C\|f\|_{L_R^{2,\delta}}$.

By the interior parabolic Schauder estimates, we obtain that for any $ t\leq 0$ and $x\in \Si_R$,
\begin{align}
    \|u_R\|_{C^{2,\alpha}(B_1(x)\times ( t-1, t))}\leq C (\|u_R\|_{L^2(B_2(x)\times( t-2, t) )}+\|f\|_{C^{0,\alpha}(B_2(x)\times ( t-2, t))}). 
\end{align}
We choose a covering of $\Si_R$ and sum to get
\begin{align}
   \|u_R\|_{C^{2,\alpha}(\Si_R\times ( t-1, t))}\leq C (\|u_R\|_{L^2(\Si_R\times( t-2, t) )}+\|f\|_{C^{0,\alpha}(\Si_R\times ( t-2, t))}).
\end{align}
Multiplying by $e^{-\delta t}$ and taking the supremum for $ t\leq 0$, we obtain
\begin{align}
   \|u_R\|_{2,\alpha,\delta;R}\leq C (\|u_R\|_{L_R^{2,\delta}}+\|f\|_{0,\alpha,\delta;R})\leq C(\|f\|_{L_R^{2,\delta}}+\|f\|_{0,\alpha,\delta;R})\leq C\|f\|_{0,\alpha,\delta; R},
\end{align}
which completes the proof.
\end{proof}

Let $\mathbf{a} = (a_1, \ldots, a_I) \in \mathbb{R}^I$. For a given integer subset $J \subset \{1, \ldots, I\}$, we define the operator $\iota_{R}^{J}: \mathbb{R}^I \to L^2(\Si_R\ti \mathbb{R}_-) $ by
\begin{align}
    \iota_{R}^{J}(\mathbf{a}) := \sum_{j \in J} a_je^{-\lambda_{R,j} t} \phi_{R,j}.
\end{align}
Recall $L=\#\{\lambda_{1},\lambda_{2},\cdots,\lambda_{I}\}$ is the number of distinct negative eigenvalues of $\mathcal L$ on $\Si$, and $\lambda'_{i}$ is the $i$-th eigenvalue of $\cL$ on $\Si$ counted without multiplicity. Take $\epsilon>0$ small enough to be determined, such that $\epsilon<\frac{1}{10}\min\{\lambda_l'-\lambda'_{l-1}|l=2,\cdots,L\}$. For each $l = 1, \cdots, L$, we define
\begin{align}
J_{R}^{(l)} =& \{j : \lambda_{R,j} \in(\lambda'_{l}-\epsilon,\lambda'_{l}+\epsilon) \}=:\{i_{R,l},i_{R,l}+1,\cdots,i_{R,l+1}-1\}.
\end{align}
By the choice of $\epsilon$, the intervals $(\lambda'_{l}-\epsilon,\lambda'_{l}+\epsilon)$ ($l=1,\cdots,L$) are disjoint. Thus $J_R^{(l)}$ forms a disjoint partition of $\{1,2,\cdots, I_R=I\}$. Since $\lambda_{R,i}\to\lambda_{i}$ as $R\to\infty$, we can choose  $R_0$ large enough, such that 
\begin{equation}
 J_R^{(l)}=J^{(l)}=\{j:\lambda_j=\lambda'_{l}\}=:\{i_l,i_l+1,\cdots,i_{l+1}-1\}   
\end{equation}
being independent of $R\geq R_0$. Moreover, we can choose $\delta_l>0$ ($l=1,2,\cdots,L$) different from each $-\lambda'_{l}$, such that
\begin{equation}
\begin{aligned}
-\delta_{l-1}<\max\{2(\lambda'_{l}+\epsilon),\lambda'_{l-1}+\epsilon\}<-\delta_{l}<\lambda'_{l}-\epsilon,\quad l=1,\cdots,L,
\end{aligned}
\end{equation}
holds for all $R\geq R_0$ large (thus for $R=\infty$), where we use the convention $-\delta_{0}=\lambda_{R,0}=-\infty$. 

We then denote $X_{R}^{(l)} = X_{R}^{\delta_{l}}$, $\iota_{R}^{(l)} = \iota_{R}^{J_{R}^{(l)}}$, and $P_{R}^{(l)} := \sum_{j : \lambda_{R,j} < -\delta_{l}} P_{R,j}$. In what follows, the notation $\lesssim$ indicates inequalities that hold up to multiplicative constants dependent on $\Si$.

\begin{theorem}\label{anc R}
There exists a number $\varepsilon_0(\Si)>0$, independent of $R\geq R_0$, such that for any $\mathbf{a}=(a_1,\ldots,a_{I})\in\mathbb{R}^{I}$ with $|\mathbf{a}|<\varepsilon_0$, we can uniquely determine a set of functions $\{u^{(l)}_R\}_{l=1}^{L}$, depending continuously on $\mathbf{a}$: For each $l = 1, \ldots, L$, the sum $\sum_{j=1}^{l} u^{(j)}_R$ forms an ancient solution of \eqref{eq:MR} satisfying $u^{(l)}_R-\iota^{(l)}_R(\mathbf{a})\in X^{(l)}_R$ and the terminal condition $P^{(l)}_R(u^{(l)}_R-\iota^{(l)}_R(\mathbf{a}))(\cdot,0)=0$. Furthermore,
\begin{equation}\label{m limit}
\lim_{ t\to-\infty}e^{\lambda_{R,m} t}(u^{(l)}_R(\cdot, t),\phi_{R,m})=a_m \quad \text{for each } m \in J^{(l)}.
\end{equation}

\end{theorem}

\begin{proof} 
We prove this theorem using induction. The argument proceeds analogously to Theorem 3.2 in \cite{CS1}.

Write $w_R = u^{(1)}_R-\iota^{(1)}_R$. Then, finding the solution $u^{(1)}$ for \eqref{eq:MR} is equivalent to finding the solution $w$ for
\begin{equation}\label{fix eq}
    \left\{
    \begin{aligned}
          \partial_ t w_R&=\mathcal{L}w_R+E(\iota^{(1)}_R(\mathbf{a})+w_R)\quad\text{ in }\Si_R\times\mathbb{R}_-;\\
          w_R&=0\quad\text{ on }\partial \Si_R\times \mathbb{R}_-.
    \end{aligned}
    \right.
\end{equation}
We will find this $w_R$ using the contraction mapping theorem.  We omit $R$ in $u_R$, $w_R$, $\iota_R$ in the proof for simplicity.

\bigskip

First, utilizing \Cref{lem:E}, we claim the following.\\
{\bf Claim 1.} There exists a small $\varepsilon_0 > 0$ such that if $\|w\|_{X^{(1)}} + \|w_{1}\|_{X^{(1)}} + \|w_{2}\|_{X^{(1)}} + |\mathbf{a}| < \varepsilon_0$, then
\begin{equation}\label{E holder}
\|E(w+\iota^{(1)}(\mathbf{a}))\|_{0,\alpha,\delta_{1}} \leq C( \|w\|^{2}_{X^{(1)}} + |\mathbf{a}|^2)
\end{equation}
\begin{equation}\label{E diff holder}
\|E(w_{1}+\iota^{(1)}(\mathbf{a}))-E(w_{2}+\iota^{(1)}(\mathbf{a}))\|_{0,\alpha,\delta_{1}} \leq C\varepsilon_0\|w_{1}-w_{2}\|_{X^{(1)}},
\end{equation}
where $C$ depends only on $\Si,\alpha,\delta_1$.

\bigskip
    
Applying \Cref{lem:E}, we have
\begin{align}
\|E(w+\iota^{(1)}(\mathbf{a}))( t)\|_{C^{0,\alpha}} \leq C\|(w+\iota^{(1)}(\mathbf{a}))( t)\|_{C^{2,\alpha}}^2 
\leq C(\|w( t)\|_{C^{2,\alpha}}^2 + |\mathbf{a}|^2e^{-2\lambda_{R,1} t}).
\end{align}
By multiplying this inequality with $e^{-\delta_{1} t}$ and then taking the supremum, and noting that $-\delta_{1} > 2\lambda_{R,1}$, we obtain \eqref{E holder}.

The second inequality \eqref{E diff holder} follows in a similar way. Indeed, by \Cref{lem:E}, we obtain
\begin{align}
 \|E(w_{1}+\iota^1(\mathbf{a}))-E(w_{2}+\iota^{(1)}(\mathbf{a}))\|_{C^{0,\alpha}}
&\leq C(\|w_{1}\|_{C^{2,\alpha}}+\|w_{2}\|_{2,\alpha}+|\mathbf{a}|e^{-\lambda_{R,1} t})\|(w_{1}-w_{2})\|_{2,\alpha}\\
&\leq C\varepsilon_0\|(w_{1}-w_{2})( t)\|_{C^{2,\alpha}}
\end{align}
since $\|w_{1}( t)\|_{C^{2,\alpha}}+\|w_{2}( t)\|_{C^{2,\alpha}}+|\mathbf{a}|\leq \|w_{1}\|_{{2,\alpha,\delta_{1}}}e^{\delta_{1} t}+\|w_{2}\|_{{2,\alpha,\delta_{1}}}e^{\delta_{R,1} t}+|\mathbf{a}|<\varepsilon_0$.
Then, by multiplying with $e^{-\delta_{1} t}$ and taking the supremum, we obtain \eqref{E diff holder}.

\bigskip

We are now ready to use the contraction mapping theorem. 
Define the map $S: \{f \in X^{(1)} : \|f\|_{X^{(1)}} < \varepsilon_0\} \to X^{(1)}$ by setting $S(w) = v$, where $v$ is the solution of
\begin{equation}
\left\{
\begin{aligned}
\partial_ t v - \mathcal{L}v = E(w + \iota^{(1)}(\mathbf{a})) \qu\text{in } \Si_R \times \mathbb{R}_-\\
v=0\quad\text{on }\partial \Si_R\times \mathbb{R}_-.
\end{aligned}
\right.
\end{equation}
with $P_R^{(1)}(v(\cdot, 0)) = 0$. By \Cref{app solu}, such $v$ is unique, making $S(w)$ well-defined. Furthermore, \Cref{app solu} and the claim imply
\begin{align}
\|S(w)\|_{X^{(1)}} &\leq C\|E(w + \iota^{(1)}(\mathbf{a}))\|_{0,\alpha,\delta_{1};R} \leq C(\|w\|^{2}_{X^{(1)}} + |\mathbf{a}|^2)
\\
\|S(w_1) - S(w_2)\|_{X^{(1)}} &\leq C\|E(w_1 + \iota^{(1)}(\mathbf{a})) - E(w_2 + \iota^{(1)}(\mathbf{a}))\|_{0,\alpha,\delta_{1}} \leq  C\varepsilon_0\|w_1 - w_2\|_{X^{(1)}}.
\end{align}

If $\varepsilon_0$ is chosen to be sufficiently small, then $S$ becomes a contraction mapping on the set $\{f \in X^{(1)} \mid \|f\|_{X^{(1)}}  < \varepsilon_0\}$. Consequently, we can find a fixed point $w$, which satisfies (\ref{fix eq}), and we define $u^{(1)} := w + \iota^{(1)}(\mathbf{a})$. 

Since $|\phi_{R,m}(x)|\leq C\phi^{(1-\varepsilon)\frac{\lambda_{R,m}}{\lambda}}(x)\leq Ce^{-c_m\rho(x)}$ for some $c_m>0$ by Lemma \ref{quo boun} and the exponential decay of $\phi$ \footnote{The results for eigenfunctions of Schr\"odinger operator in Euclidean space is classical, see \cite{A} or Chapter 3 of \cite{HS} for details. The proof for a complete manifold $M$ with metric $g_{ij}$ close to the standard Euclidean metric and linear decay for the potential function $|A|^2$ (Remark \ref{A bound g alpha}) follows almost the same.}, $\int_\Si|\phi_{R,m}| \, dvol_\Si\leq C$ for some uniform constant independent of $m$ and $R$. On the other hand, since $|w(\cdot, t)| \leq \|w\|_{0,\delta_{1}} e^{\delta_{1} t} \leq \varepsilon_0 e^{\delta_{1} t}$, we have
\begin{equation}
\lim_{ t \to -\infty} e^{\lambda_{R,m} t}(u^{(1)}, \phi_{R,m})  = \lim_{ t \to -\infty} \varepsilon_0e^{(\delta_{1} + \lambda_{R,m}) t} + a_m = a_m, \quad m \in J_R^{(1)}.
\end{equation}
Thus, we have found the desired $u^{(1)}$.

Suppose that we have found $u^{(1)}_R$ up to $u^{(l)}_R$. If $J^{(l+1)}_R=\emptyset$, we are done. Otherwise, let $u^{(l+1)}_R=w+\iota^{(l+1)}(a)_R$, it suffices to require
\begin{equation}
\left\{
\begin{aligned}
    \partial_ t w_R=\mathcal Lw_R+E^{(l+1)}(w_R)\\
    w_R=0\quad\text{ on }\partial \Si_R\times \mathbb{R}_-,
\end{aligned}
\right.
\end{equation}
where
\begin{align}
    E^{(l+1)}(w_R)=E(w+\iota ^{(l+1)}_R(\bfa)+\sum_{j=1}^l u^{(j)}_R)-E(\sum_{j=1}^l u^{(j)}_R).
\end{align}
As before, we omit $R$ in the proof if there is no confusion.

{\bf Claim 2.} There exists $\varepsilon_0$ small such that if $\|w\|_{X^{(l+1)}}+\|w_1\|_{X^{(l+1)}}+\|w_2\|_{X^{(l+1)}}+|\bfa|<\varepsilon_0$, then
\begin{align}\label{El hold}
    \|E^{(l+1)}(w)\|_{{0,\alpha,\delta_{l+1}}}&\leq C(\|w\|^2_{X^{(l+1)}}+\varepsilon_0^2)\\
\label{El diff hold}
    \|E^{(l+1)}(w_1)-E^{(l+1)}(w_2)\|_{{0,\alpha,\delta_{l+1}}}&\leq C\varepsilon_0\|w_1-w_2\|_{X^{(l+1)}}.
\end{align}
Indeed, since $J^{(l+1)}=\{i_{l+1},i_{l+1}+1,\cdots,i_{l+2}-1\}$, it follows from \eqref{diff holder} and $u^{(j)}\in X^{(\delta_{j})}$ with $\|u^{(j)}\|_{X^{\delta_{j}}}<\varepsilon_0$ ($j=1,2,\cdots,l$) that
\begin{equation}
    \begin{aligned}
        \|E^{(l+1)}(w)( t)\|_{C^{0,\alpha}}%\leq& C\|w( t)+\iota^{(l+1)}(\bfa)+2\sum_{j=1}^lu^{(j)}\|_{C^{2,\alpha}}\|w( t)+\iota^{(l+1)}(\bfa)\|_{C^{2,\alpha}}\\
        \leq& C(\|w( t)\|_{C^{2,\alpha}}+|\bfa|e^{-\lambda_{R,i_{l+2}-1} t}+\varepsilon_0e^{\delta_l t})(\|w( t)\|_{C^{2,\alpha}}+|\bfa|e^{-\lambda_{R,i_{l+2}-1} t})\\
        \leq& C(\|w( t)\|_{C^{2,\alpha}}+\varepsilon_0e^{-\lambda_{R,i_{l+2}-1} t})^2\quad (\text {since } \delta_l>-\lambda_{R,i_{l+2}-1})\\
        \leq& Ce^{-2\lambda_{R,i_{l+2}-1} t}(\|w\|_{{2,\alpha,-\lambda_{R,i_{l+2}-1}}}+\varepsilon_0)^2\\
        % \leq& Ce^{-2\lambda'_{R,{l+1}} t}(\|w\|_{{2,\alpha,\delta_{l+1}}}+\varepsilon_0)^2\\
        \leq& Ce^{-2\lambda_{R,i_{l+2}-1} t}(\|w\|^2_{{2,\alpha,\delta_{l+1}}}+\varepsilon_0^2),\quad  (\text {since } \delta_{l+1}>-\lambda_{R,i_{l+2}-1}).
    \end{aligned}
\end{equation}
Since $\max\{2\lambda_{R,i_{l+2}-1},\lambda_{R,i_{l+1}-1}\}<-\delta_{l+1}$, this implies that
\begin{equation}
 \|E^{(l+1)}(w)\|_{{0,\alpha,\delta_{l+1}}}\leq C(\|w\|^2_{{2,\alpha,\delta_{l+1}}}+\varepsilon_0^2)=C(\|w\|^2_{X^{(l+1)}}+\varepsilon_0^2),
\end{equation}
which proves \eqref{El hold}.

On the other hand, since $\delta_l>\delta_{l+1}>-\lambda_{R,i_{l+2}-1}$, we obtain
\begin{equation}
\begin{aligned}
    &\fr{\|E^{(l+1)}(w_1)( t)-E^{(l+1)}(w_2)( t)\|_{C^{0,\alpha}}}{\|w_1( t)-w_2( t)\|_{C^{2,\alpha}}}
    %=&\|E(w_1+\iota^{(l+1)}(a)+\sum_{j=1}^lu^{(j)})( t)-E(\sum_{j=1}^lu^{(j)})( t)-E(w_2+\iota^{(l+1)}(a)+\sum_{j=1}^lu^{(j)})( t)+E(\sum_{j=1}^lu^{(j)})( t)\|_{C^{0,\alpha}}\\
    \leq \|w_1( t)+w_2( t)+2\iota^{(l+1)}(\bfa)( t)+2\sum_{j=1}^lu^{(j)}( t)\|_{C^{2,\alpha}}\\
    %\leq& C(\|w_1( t)\|_{C^{2,\alpha}}+\|w_2( t)\|_{C^{2,\alpha}}+\|\iota^{(l+1)}(\bfa)( t)\|_{C^{2,\alpha}}+\sum_{j=1}^l\|u^{(j)}( t)\|_{C^{2,\alpha}})\|w_1( t)-w_2( t)\|_{C^{2,\alpha}}\\
    &\leq C(\|w_1\|_{X^{(l+1)}}e^{\delta_{l+1} t}+\|w_2\|_{X^{(l+1)}}e^{\delta_{l+1} t}+|\bfa|e^{-\lambda_{R,i_{l+2}-1} t}+\varepsilon_0e^{\delta_{l} t})
    \leq C\varepsilon_0e^{-\lambda_{R,i_{l+2}-1}  t}
    \leq C\varepsilon_0,
\end{aligned}
\end{equation}
which proves \eqref{El diff hold}.

Now, we can define a map $S:\{f\in X^{(l+1)}:\|f\|_{X^{(l+1)}}<\varepsilon_0\}\to X^{(l+1)}$ by $S(w)=v$, where $v$ is defined as the unique solution of
\begin{equation}
    \partial_ t v-\mathcal L v=E^{(l+1)}(w)\quad\text{ on }\Si_R\times\mathbb R_-,
\end{equation}
with $P^{(l+1)}(u(\cdot,0))$=0. Taking $\varepsilon_0$ small enough, $S$ is a contraction mapping. Since $w^{(l+1)}\in X^{(\delta_{l+1})}$, and $\delta_{l+1}>-\lambda_{R,i_{l+2}-1}$, we have \eqref{m limit} holds for $m\in J^{(l+1)}$. The existence of $u^{(l+1)}$ is established. Finally, the uniqueness and continuity of $u^{(l+1)}$ also follow from the contraction mapping theorem.
\end{proof}

%subsection
\subsection{Ancient solutions on $\Si$}

%paragraph
In this subsection, we construct ancient solutions on \(\Sigma\) using the approximate solutions on \(\Sigma_R\) obtained in the previous subsection.

\begin{theorem}\label{thm:exist}
There exists a $\varepsilon_0(\Si)>0$ such that for any $\mathbf{a}=(a_1,\ldots,a_I)\in\mathbb{R}^I$ with $|\mathbf{a}|<\varepsilon_0$, we can uniquely determine a set of functions $\{u^{(l)}\}_{l=1}^{L}$, depending continuously on $\mathbf{a}$: For each $l = 1, \ldots, L$, the graph of sum $\sum_{j=1}^{l} u^{(j)}$ forms an ancient solution of the mean curvature flow. Moreover, we have $u^{(l)}-\iota^{(l)}(\mathbf{a})\in X^{(l)}$ and the terminal condition $P^{(l)}(u^{(l)}-\iota^{(l)}(\mathbf{a}))(\cdot,0)=0$. Furthermore,
\begin{equation}
\lim_{ t\to-\infty}e^{\lambda_m t}(u^{(l)}(\cdot, t),\phi_m)=a_m \quad \text{for each } m \in J^{(l)}.
\end{equation}
\end{theorem}
% $M^n$ denotes a complete minimal surface in $\mathbb{R}^{n+1}$ with finite total curvature, and $I$ is the Morse index of $\cL=\Delta + |A|^2$ for $M$.
\begin{proof}
    
Consider the ancient solutions $u^{(l)}_{R}$ of (\ref{eq:MR}) constructed in \Cref{anc R}. We first note that since $w_R^{l}\equiv 0$ on $(\Si\setminus \Si_R)\times\mathbb R_-$, which implies that $\|w^{(l)}_R\|_{X^\delta_R}=\|w^{(l)}_R\|_{X^\delta}$.  For $l=1$, we have $u_{R}^{(1)} = w_R + \iota_R^{(1)}(\mathbf{a})$, where $\|w_R\|_{{2,\alpha,\delta_{1}}} < \varepsilon_0$. This and the boundedness of $\phi_{R,1}$ (Lemma \ref{quo boun}) implies 
\begin{align}
\|u_R^{(1)}\|_{{2,\alpha}} &\leq \|w_R( t)\|_{{2,\alpha}} + \|\iota_R^{(1)}(\mathbf{a})\|_{{2,\alpha}}\leq\varepsilon_0e^{\delta_{1} t}+C|\mathbf a|e^{-\lambda_{R,1}  t}\leq C\varepsilon_0, 
%\\ &\leq \varepsilon_0 e^{\delta_{R,1}  t} + C(n)|\mathbf{a}|e^{-\lambda{R,I}  t} \leq C(n)\varepsilon_0 e^{-\lambda_{R,I}  t} \leq C(n)\varepsilon_0 e^{-\frac{1}{2}\lambda_I  t}
\end{align} 
for sufficiently large $R\geq R_0$, since $\delta_{1}>-\lambda_{R,1}$ and $|\mathbf a|<\varepsilon_0$. Thus, we can extract a convergent subsequence $u_{R_j}^{(1)}$ from $u_R^{(1)}$, which converges in ${C}_{loc}^{2,\alpha}(\Si\times\mathbb{R}_-)$ to a function $u^{(1)}$ as $R_j \to \infty$ (as $j \to \infty$) and $u^{(1)}$ satisfies
\begin{equation}
    \partial_ t u^{(1)} = \mathcal{L}u^{(1)} + E(u^{(1)}) \quad \text{on } \Si \times \mathbb R_-.
\end{equation}
Note that $\ind(\Si_{R_k})=\ind(\Si)$ for sufficiently large $R_k$, $\lambda_{R_k,j} \searrow \lambda_j$ and $\phi_{R_k,j} \to \phi_j$ in $C^\infty_{loc}(\Si)$ as $k \to \infty$. It follows that $\iota_R^{(1)}(\mathbf{a})$ converges to $\sum_{j\in J^{(1)}}a_je^{-\lambda_j t}\phi_j$ in ${C}_{loc}^{2,\alpha}(\Si\times\mathbb{R}_-)$. Consequently, defining $w^{(1)} = u^{(1)} - \iota^{(1)}(\mathbf{a}) $, we have  $w_{R_k}^{(1)}\to w^{(1)}$ in the same space. Since $\|w_{R_k}^{(1)}( t)\|_{C^{2,\alpha}} \leq \varepsilon_0 e^{\delta_1 t}$, we have $\|w^{(1)}( t)\|_{C^{2,\alpha}(\Si)}\leq \varepsilon_0 e^{\delta_1 t}$. 
%\marginpar{specify how to choose $\de_1^R$. Maybe $\de_1^R=\de_1$?}

On the other hand, the exponential decay of $\phi$, Lemma \ref{quo boun}, and $\phi_{R,m}\to \phi_m$ implies the exponential decay of $\phi_m$ and thus $\|\phi_m\|_{L^1(M)}<\infty$. We have
\begin{equation}
\lim_{ t \to -\infty} e^{\lambda_m  t} (u^{(1)}, \phi_m) 
=\lim_{ t \to -\infty} e^{\lambda_m  t} (\iota^{(1)}(\bfa), \phi_m)
     =a_m, \quad m \in J^{(1)}.
\end{equation}
This completes the proof for the case $l=1$.

The case for \(l \geq 2\) can be obtained similarly.
\end{proof}

%subsection

\section{Space-time asymptotics of ancient solutions, mean convexity, and finite mass drop}\label{sec asymp}

\subsection{Space-time decay and finite mass drop of ancient solutions}
Our goal in this subsection is to show the space-time decay in Theorem \ref{thm:main.I_parameter}. Specifically, we prove the following:

\begin{theorem}\label{thm:decay and finite mass}
    Let $\{u^{(l)}\}_{l=1}^{L}$ be the functions in Theorem \ref{thm:exist}. Given a small $\varepsilon>0$, there exists $C(\varepsilon,M)>0$, independent of $(x,t)\in M\times\mathbb R_-$, such that for any $l\in\{1,\cdots,L\}$, we have
    \begin{equation}\label{uk alpha}
    \|u^{(l)}\|_{C^{2,\alpha}(B_1(x)\times ( t-1, t))}\leq C|\bfa|e^{-2\lambda'_l t}\phi(x)^{(1-\varepsilon)\frac{\lambda'_l}{\lambda}}, \quad (x, t)\in \Si\times \mathbb R_-.
\end{equation}
In particular, combined with exponential decay of $\phi$, we have for $u=\sum_{l=1}^{L}u^{(l)}$ that 
\begin{equation}
\begin{aligned}
     \int_{-\infty}^0\int_{M_ t} H^2 dvol_{M_ t}d t\leq&C|\bfa|^2\int_{-\infty}^0\int_{\mathbb{R}^n} \sum_{l=1}^{L}e^{-2\lambda'_l t}e^{-\varepsilon'_l|y|}dy  d t<\infty
\end{aligned}
\end{equation}
and
\begin{equation}
    \int_{M_t}|A|^ndvol_{M_t}<Ce^{-\varepsilon't}<\infty,\quad t\in\mathbb R_-,
\end{equation}
for some $\varepsilon_l',\varepsilon'>0$.
\end{theorem}

Before we proceed, let us introduce a notation. For an integer $m\geq 0$ and $\alpha\in (0,1)$, and for any positive real numbers $r,s>0$, we define
\begin{equation}
    %\|f\|_{C^{m,\alpha}_{r,s}(x,t)}:=
    \|f\|_{C^{m,\alpha}_{r,s}}(x,t):=\|f\|_{C^{m,\alpha}(B_r(x)\times (t-s,t))}, \quad (x,t)\in \Si\times\mathbb R_-
\end{equation}
for any function $f\in C^{m,\alpha}(\Si\times\mathbb R_-)$. If $x,t$ is fixed, then we simply write $\|f\|_{C^{m,\alpha}_{r,s}}$ for $\|f\|_{C^{m,\alpha}_{r,s}}(x,t).$ % $\|f\|_{C^{m,\alpha}_{r,s}(x,t)}.$

\begin{lemma}
Let $\{u_R^{(k)}\}_{k=1}^L$ be the functions obtained in Theorem \ref{anc R}, then there exists $R_0>0$ large, such that for any $\varepsilon>0$, there exists $C(\varepsilon,M)>0$ independent of $R\geq R_0$, $(x,t)\in M\times\mathbb R_-$, such that
\begin{equation}\label{v k holder weighted}
    \|u^{(k)}_R\|_{C^{2,\alpha}_{1,1}}(x,t)\leq C|\bfa|e^{-2\lambda_{R,i_{k+1}-1} t}\phi^{(1-\varepsilon)\frac{\lambda_{R,i_{k+1}-1}}{\lambda}}(x),\quad k=0,1,2,\cdots,L,
\end{equation}
where we use the convention that $u_R^{(0)}=0$. 
\end{lemma}

\begin{proof}
We prove this by induction on $k$. Since $u_R^{(0)}\equiv 0$, \eqref{v k holder weighted} holds for $k=0$. Assuming the statement is true for $k=0,\cdots,l$, we prove it for $k=l+1$.

Recall that $u^{(l+1)}_R=w_R+\iota^{(l+1)}_R(a)$ and $\partial_ t w_R=\mathcal Lw_R+E^{(l+1)}_R(w_R)$,    
where
\begin{align}
    E^{(l+1)}_R(w_R)=E(w_R+\iota ^{(l+1)}_R(\bfa)+\sum_{j=1}^l u^{(j)}_R)-E(\sum_{j=1}^l u^{(j)}_R),
\end{align}
and we find $w_R$ by the contraction mapping $S:\mathcal A:=\{f\in X^{(l+1)}:\|f\|_{X^{(l+1)}}<\varepsilon_0\}\to X^{(l+1)}$, where we use the convention that $E(0)=0$. Writing $w_R^{(b+1)}=S(w_R^{(b)})$, $b=1,2,\cdots$, then $w_R=\lim_{b\to\infty}w_R^{(b)}$.   

We start with $b=0$. Since $w_R$ is unique, we can start with any $w_R^{(0)}\in \mathcal A$. For simplicity, we let $w_R^{(0)}=0$. Then
\begin{align}
    E^{(l+1)}_R(w_R^{(0)})
     =E(\iota ^{(l+1)}_R(\bfa)+\sum_{j=1}^l u^{(j)}_R)-E(\sum_{j=1}^l u^{(j)}_R).
\end{align}
By \eqref{diff holder} in Lemma \ref{lem:E}, we have
\begin{align}
     \|E^{(l+1)}_R(w_R^{(0)})\|_{C^\alpha_{1,1}}
     %\leq& C\|\iota ^{(l+1)}_R(a)\|_{C^{2,\alpha}(B_1(x)\times( t-1, t))}(\|\iota ^{(l+1)}_R(a)+\sum_{j=1}^l u^{(j)}_R\|_{C^{2,\alpha}(B_1(x)\times( t-1, t))}+\|\sum_{j=1}^l u^{(j)}_R\|_{C^{2,\alpha}(B_1(x)\times( t-1, t))})\\
     &\leq C\|\iota ^{(l+1)}_R(\bfa)\|_{C^{2,\alpha}_{1,1}}(\|\iota ^{(l+1)}_R(\bfa)\|_{C^{2,\alpha}_{1,1}}+\|\sum_{j=1}^l u^{(j)}_R\|_{C^{2,\alpha}_{1,1}})\\
     &=C\|\iota ^{(l+1)}_R(\bfa)\|^2_{C^{2,\alpha}_{1,1}}+C\|\iota ^{(l+1)}_R(\bfa)\|_{C^{2,\alpha}_{1,1}}\|\sum_{j=1}^l u^{(j)}_R\|_{C^{2,\alpha}_{1,1}}=:I_1+I_2.
\end{align}
 Note that $J_R^{(l+1)}=J^{(l+1)}=\{i_{l+1},i_{l+1}+1,\cdots,i_{l+2}-1\}$ and $\lambda_{R,i}\in(\lambda'_{l+1}-\epsilon,\lambda'_{l+1}+\epsilon)$ ($i\in J_R^{(l+1)}$) for sufficiently large $R$. By Theorem \ref{2 alpha point quo}, we obtain
 \begin{equation}\label{iota1}
 \begin{aligned}
     I_1\leq C|\bfa|^2e^{-2\lambda_{R,i_{l+2}-1} t}\sum_{i\in J^{(l+1)}}\|\phi_{R,i}\|^2_{C^{2,\alpha}(B_1(x))}
     \leq C|\bfa|^2e^{-2\lambda_{R,i_{l+2}-1} t}\phi^{2(1-\varepsilon)\frac{\lambda_{R,i_{l+2}-1}}{\lambda}}(x).
    % \leq &C|a|^2e^{-2\lambda_{R,i_1} t}\phi^{2(1-\varepsilon)\frac{l+1}{L_R}}(x)
 \end{aligned}  
 \end{equation}
 On the other hand, by the induction hypothesis \eqref{v k holder weighted} for $k=1,2\cdots,l$, we have
\begin{equation}\label{iota1 u0}
 \begin{aligned}
     I_2
     \leq&C|\bfa|^2e^{-(\lambda_{R,i_{l+2}-1}+\lambda_{R,i_{l+1}-1}) t}\sum_{i\in J^{(l+1)}}\|\phi_{R,i}\|_{C^{2,\alpha}(B_1(x))}\sum_{i=1}^{i_{l+1}-1}\phi^{(1-\varepsilon)\frac{\lambda_{R,i}}{\lambda}}(x)\\
     \leq &C|\bfa|^2e^{-(\lambda_{R,i_{l+2}-1}+\lambda_{R,i_{l+1}-1}) t}\phi^{(1-\varepsilon)\frac{\lambda_{R,i_{l+2}-1}+\lambda_{R,i_{l+1}-1}}{\lambda}}(x).
     %\leq &C|a|^2e^{-(\lambda_{R,i_{l+1}}+\lambda_{R,i_l}) t}\phi^{(1-\varepsilon)\frac{l+1}{L_R}}(x)\phi^{(1-\varepsilon)(\frac{l+2}{L_R})}(x)
 \end{aligned}  
 \end{equation}
Therefore, it follows from $\lambda_{R,i_{l+1}-1}<\lambda_{R,i_{l+2}-1}<0$ that
\begin{equation}\label{error l phi holder}
     \|E^{(l+1)}_R(w_R^{(0)})\|_{C^{\alpha}_{1,1}}\leq C|\bfa|^2e^{-2\lambda_{R,i_{l+2}-1} t}\phi^{2(1-\varepsilon)\frac{\lambda_{R,i_{l+2}-1}}{\lambda}}(x).
 \end{equation}
%for $b=1,2,\cdots$ and  $l=0,1,2,\cdots, L$.

\bigskip

Now we derive the H\"older type estimate (\ref{v k holder weighted}) for $w^{(1)}_R$. Set $\psi=\phi^m$, for $m>0$ to be determined. Then we have
\begin{equation}
    \begin{aligned}
        \nabla\psi&=m\phi^{m-1}\nabla\phi,\\
        \Delta\psi&=m(m-1)\phi^{m-2}|\nabla\phi|^2+m\phi^{m-1}\Delta\phi=m(m-1)\phi^{m-2}|\nabla\phi|^2-m\phi^{m}(|A|^2+\lambda)\\
        &=\tfrac{m-1}{m}\tfrac{|\nabla\psi|^2}{\psi}-m(|A|^2+\lambda)\psi.
    \end{aligned}
\end{equation}
Let $\hat w_R^{(b+1)}=e^{\mu t}\frac{w_R^{(b+1)}}{\psi}$ ($\mu>0$ to be determined), we obtain
\begin{align}
    \hat w^{(b+1)}_{R, t}&=\mu \hat w^{(b+1)}_R+e^{\mu t}\tfrac{w_{R, t}^{(b+1)}}{\psi}=\mu\hat w_R^{(b+1)}+e^{\mu t}\tfrac{\mathcal Lw_R^{(b+1)}+E^{(l+1)}(w_R^{(b)})}{\psi}\\
    &=\mu\hat w_R^{(b+1)}+\tfrac{(\Delta+|A|^2)(\psi\hat w_R^{(b+1)})}{\psi}+e^{\mu t}\tfrac{E^{(l+1)}(w_R^{(b)})}{\psi}\\
    &=\Delta \hat w_R^{(b+1)}+2\tfrac{\nabla\psi}{\psi}\cdot\nabla\hat w_R^{(b+1)}+(\mu+(1-m)|A|^2-m\lambda+\tfrac{m-1}{m}\tfrac{|\nabla\psi|^2}{\psi^2})\hat w_R^{(b+1)}+e^{\mu t}\tfrac{E^{(l+1)}(w_R^{(b)})}{\psi}.
\end{align}

We take $\mu=\lambda'_{{l+1}}$ and separate into two cases:
\\
{\bf Case 1:} $l+1=1$. We take $m=1< 2(1-\varepsilon)\frac{\lambda_{R,1}}{\lambda_1}$ (for $R\geq R_0$ large). Then $\mu=\lambda_1=\lambda$, and
\begin{equation}
      \hat{w}_{R, t}^{(b+1)}=\Delta \hat w_R^{(b+1)}+2\tfrac{\nabla\psi}{\psi}\cdot\nabla\hat w_R^{(b+1)}+e^{\mu t}\tfrac{E^{(l+1)}(w_R^{(b)})}{\phi^{m}}.
\end{equation}
\\
{\bf Case 2:} $l+1\geq 2$. We take  $0<\sigma<1$, $R\geq R_0$ large, such that $(1-\varepsilon)\frac{\lambda_{R,i_{l+2}-1}}{\lambda}\leq m:=(1-\sigma)\frac{\lambda'_{l+1}}{\lambda}<\min\{2(1-\varepsilon)\frac{\lambda_{R,i_{l+2}-1}}{\lambda},1\}$. Then 
\begin{equation}
    \begin{aligned}
        \hat{w}_{R, t}^{(b+1)}=\Delta \hat w_R^{(b+1)}+2\tfrac{\nabla\psi}{\psi}\cdot\nabla\hat w_R^{(b+1)}+(\sigma\lambda'_{l+1}+(1-m)|A|^2-m(1-m)\tfrac{|\nabla\phi|^2}{\phi^2})\hat w_R^{(b+1)}+e^{\mu t}\tfrac{E^{(l+1)}(w_R^{(b)})}{\phi^{m}}.
    \end{aligned}
\end{equation}
Due to the fact that $m\in [0,1]$, $\sigma\in (0,1)$, $\lambda'_{l+1}<0$, and the quadratic decay of $|A|$ by Remark \ref{A bound g alpha}, we have for $R_0$ large,
$$\sigma\lambda'_{{l+1}}+(1-m)|A|^2-m(1-m)\tfrac{|\nabla\phi|^2}{\phi^2}\leq 0,\quad x\in \Si\setminus \Si_{R_0}$$
In summary, we have for both cases
\begin{equation}\label{El+10}
    \hat{w}_{R, t}^{(b+1)}=\Delta \hat w_R^{(b+1)}+2\tfrac{\nabla\psi}{\psi}\cdot\nabla\hat w_R^{(b+1)}+c_k\hat w_R^{(b+1)}+e^{\mu t}\tfrac{E^{(l+1)}(w_R^{(b)})}{\phi^{m}},\quad x\in \Si\setminus \Si_{R_0}.
\end{equation}
with $c_k\leq 0$ on $\Si\setminus \Si_{R_0}$.

Applying Lemma \ref{lem hat error quo} with $E=E^{(l+1)}(w_R^{(0)})$, using \eqref{error l phi holder}, and note that $m\leq 2(1-\varepsilon)\frac{\lambda_{R,i_{l+2}-1}}{\lambda}$, we have 
\begin{equation}
    \|E^{(l+1)}\phi^{-m}(s)\|_{C^0_{1,1}}(x,t)\leq  C|\bfa|^2 e^{-2\lambda_{R,i_{l+2}-1} t}\phi^{2(1-\varepsilon)\frac{\lambda_{R,i_{l+2}-1}}{\lambda}-m}(x)\leq C|\bfa|^2 e^{-2\lambda_{R,i_{l+2}-1} t}
\end{equation}
for $C$ independent of $(x,t)\in M\times\mathbb R_-$.

Now, fix $ t\leq 0$. The maximum principle, and the fact that $w_R^{(1)}=\hat w_R^{(1)}e^{-\mu t}\psi\in X^{(l+1)}$ show that for $ t_R\leq  t$,
%and let $\tilde{w}_{R}=$
\begin{equation}\label{hat sup l}
\begin{aligned}
%\|\hat w^\pm_{R}( t)\|_{L^\infty}(B_2(x))\\
    \|\hat{w}_R^{(1)}( t)\|_{L^\infty(\Si_R)}
    \leq& \|\hat w_R^{(1)}( t_R)\|_{L^\infty(\Si_R)}
     + \int_{ t_R}^{ t} e^{\mu s} \|E^{(l+1)}\phi^{-m}(s)\|_{{L^\infty}(\Si_R)} ds,\\
     \leq& C\tfrac{e^{(\lambda'_{{l+1}}+\delta_{l+1}) t_R}}{c_R}+\int_{-\infty}^{ t} C|\bfa|^2e^{(\lambda'_{{l+1}}-2\lambda_{R,i_{l+2}-1})s}ds\\
     \leq& \tfrac{e^{(\lambda'_{{l+1}}+\delta_{l+1}) t_R}}{c_R}+C|\bfa|^2e^{(\lambda'_{{l+1}}-2\lambda_{R,i_{l+2}-1}) t},%\\
     %\leq& Ce^{(1-\diagma')(\lambda_i+\delta_l)}|a|^2 
     %\qu t\in [ t_R, t],
\end{aligned}
\end{equation}
where $c_R=\inf_{\Si_R}\psi=\inf_{\Si_R}\phi^m>0$. Since $\delta_{l+1}>-\lambda'_{{l+1}}$, letting $ t_R\to-\infty$, we get
\begin{equation}
      \|\hat{w}_R( t)\|_{L^\infty(\Si_R)}\leq C|\bfa|^2e^{(\lambda'_{{l+1}}-2\lambda_{R,i_{l+2}-1}) t}.
\end{equation}
Thus
\begin{align}
    \|w_R^{(1)}\|_{C^0_{1,1}}=&\|\hat w_R^{(1)}\psi e^{-\mu t}\|_{C^0_{1,1}}\leq C|\bfa|^2e^{(\lambda'_{l+1}-2\lambda_{R,i_{l+2}-1}-\mu) t}\psi(x)e^{\sup|\nabla\log \psi|}\leq C|\bfa|^2e^{-2\lambda_{R,i_{l+2}-1} t}\psi(x),
\end{align}
where we used Theorem \ref{harnack} and that $|\nabla\log\psi|=m|\nabla\log\phi|\leq|\nabla\log\phi|\leq C$ since $m\in[0,1]$. Since $\partial_ t w_R^{(1)}=\mathcal Lw_R^{(1)}+E^{(l+1)}_R(w_R^{(0)})$, the interior Schauder estimate shows that there exists $C$ independent of $(x,t)\in M\times\mathbb R_-$ such that
\begin{equation}\label{w1 inter point  schau}
\begin{aligned}
     \|w_R^{(1)}\|_{C^{2,\alpha}_{\frac{1}{2},\frac{1}{2}}}\leq &C(\|w_R^{(1)}\|_{C^0_{1,1}}+\|E^{(l+1)}(w^{(0)}_R)\|_{C^\alpha_{1,1}})\\
     \leq&C(|\bfa|^2e^{-2\lambda_{R,i_{l+2}-1} t}\psi(x)+|\bfa|^2e^{-2\lambda_{R,i_{l+2}-1} t}\phi^{2(1-\varepsilon)\frac{\lambda_{R,i_{l+2}-1}}{\lambda}}(x))\\
    \leq& C|\bfa|^2e^{-2\lambda_{R,i_{l+2}-1} t}\phi^m(x)
\end{aligned}
\end{equation}
since $m\leq 2(1-\varepsilon)\frac{\lambda_{R,i_{l+2}-1}}{\lambda}$  by the definition of $m$ and $\phi(x)$ is bounded. Since $C$ is independent of $(x,t)\in M\times\mathbb R_-$ in \eqref{w1 inter point  schau}, we can cover $B(x,1)\times(t-1,1)$ by balls $\{(z,\frac{1}{2})\times (s-\frac{1}{2},s)\}_{(z,s)\in B(x,1)\times (t-1,t)}$, and get
\begin{equation}
\begin{aligned}
    \|w_R^{(1)}\|_{C^{2,\alpha}_{1,1}}(x,t)\leq& \sup_{(z,s)\in B(x,1)\times (t-1,t)}\|w_R^{(1)}\|_{C^{2,\alpha}_{\frac{1}{2},\frac{1}{2}}}(z,s)\leq \sup_{(z,s)\in B(x,1)\times (t-1,t)} C|\bfa|^2e^{-2\lambda_{R,i_{l+2}-1} s}\phi^m(z)\\
    \leq& C |\bfa|^2e^{-2\lambda_{R,i_{l+2}-1} t}\phi^m(x),
\end{aligned}
\end{equation}
where in the last inequality, we use $m\in [0,1]$ and the Harnack inequality for $\phi$ in Theorem \ref{harnack} to get $\phi(z)\leq C\phi(x)$ for $z\in B(x,1)$.

Repeating this process with $b=1$, we have 
\begin{align}
     E^{(l+1)}_R(w_R^{(1)})=&E(w_R^{(1)}+\iota ^{(l+1)}_R(\bfa)+\sum_{j=1}^l u^{(j)}_R)-E(\sum_{j=1}^l u^{(j)}_R).%\\
     %=&E(\iota ^{(l+1)}_R(a)+\sum_{j=1}^l v^{(j)}_R)-E(\sum_{j=1}^l v^{(j)}_R).
\end{align}
Observe that
\begin{align}
\|\sum_{j=1}^l u^{(j)}_R\|^2_{C^{2,\alpha}_{1,1}}
&\le
C|\bfa|^2e^{-2\lambda_{R,i_{l+1}-1} t}\phi^{2(1-\varepsilon)(\frac{\lambda_{R,i_{l+1}-1}}{\lambda})}
\\
\|w_R^{(1)}\|^2_{C^{2,\alpha}_{1,1}}
&\le
C|\bfa|^4e^{-4\lambda_{R,i_{l+2}-1} t}\phi^{2m}(x).
\end{align}
By \eqref{diff holder} in Lemma \ref{lem:E} and Theorem \ref{2 alpha point quo}, we have
\begin{align}
     \|E^{(l+1)}_R(w_R^{(1)})\|_{C^\alpha_{1,1}}
     %\leq C(\|w_R^{(1)}\|_{C^{2,\alpha}_{1,1}}+\|\iota ^{(l+1)}_R(\bfa)\|_{C^{2,\alpha}_{1,1}})(\|w_R^{(1)}\|_{C^{2,\alpha}_{1,1}}+\|\iota ^{(l+1)}_R(\bfa)\|_{C^{2,\alpha}_{1,1}}+\|\sum_{j=1}^l u^{(j)}_R\|_{C^{2,\alpha}_{1,1}})\\
     &\leq C|\bfa|^4e^{-4\lambda_{R,i_{l+2}-1} t}\phi^{2m}(x)+C|\bfa|^2e^{-2\lambda_{R,i_{l+2}-1} t}\phi^{2(1-\varepsilon)\frac{\lambda_{R,i_{l+2}-1}}{\lambda}}(x)
     \\
     &\leq
     C|\bfa|^2e^{-2\lambda_{R,i_{l+2}-1} t}\phi^{2(1-\varepsilon)\frac{\lambda_{R,i_{l+2}-1}}{\lambda}}(x),
\end{align}
since $m\geq (1-\varepsilon)\frac{\lambda_{R,i_{l+2}-1}}{\lambda}$. 

Using the same approach as above, we can obtain
\begin{align}
     \|w_R^{(b)}\|_{C^{2,\alpha}_{1,1}}\leq C|\bfa|^2e^{-2\lambda_{R,i_{l+2}-1} t}\phi^m(x),\quad b=1,2,3,\cdots.
\end{align}
Taking $b\to\infty$, we get
\begin{equation}\label{wr schau phi}
     \|w_R\|_{C^{2,\alpha}_{1,1}}\leq C|\bfa|^2e^{-2\lambda_{R,i_{l+2}-1} t}\phi^m(x).
\end{equation}
Note that
\begin{equation}
\begin{aligned}
\|\iota^{(l+1)}_R(\bfa)\|_{C^{2,\alpha}_{1,1}}
% \leq& C|\bfa|^2e^{-2\lambda'_{R,l+1} t}\phi^m(x)+C|\bfa|\sum_{i\in J^{(l+1)}_R}\|e^{-\lambda_{R,i} t}\phi_{R,i}\|_{C^{2,\alpha}_{1,1}}\\
\leq& C|\bfa|\sum_{i\in J^{(l+1)}}e^{-\lambda_{R,i_{l+2}-1} t}\|\phi_{R,i}\|_{C^{2,\alpha}(B_1(x))}
\leq C|\bfa|e^{-\lambda_{R,i_{l+2}-1} t}\phi^{(1-\varepsilon)\frac{\lambda_{R,i_{l+2}-1}}{\lambda}}(x).%\\
\end{aligned} 
\end{equation}
These together give that
\begin{equation}
\begin{aligned}
   \|u^{(l+1)}_R\|_{C^{2,\alpha}_{1,1}}=\|\iota_R^{(l+1)}+w_R^{(b)}\|_{C^{2,\alpha}_{1,1}}
    \leq C|\bfa|e^{-\lambda_{R,i_{l+2}-1} t}\phi^{(1-\varepsilon)\frac{\lambda_{R,i_{l+2}-1}}{\lambda}}(x),%\\
%    \leq& C|a|\phi^{(1-\varepsilon)\frac{\lambda_{R,l+1}}{\lambda}}(x),
\end{aligned} 
\end{equation}
since $m\geq (1-\varepsilon)\frac{\lambda_{R,i_{l+2}-1}}{\lambda}$. % This proves \eqref{v k holder weighted} for $k=l+1$. Since $\lambda_{R,i_{l+2}}$
\end{proof}
\begin{proof}[Proof of \Cref{thm:decay and finite mass}]
Let $R\to\infty$ in \eqref{v k holder weighted}. We get that 
\begin{equation}%\label{uk alpha}
    \|u^{(k)}\|_{C^{2,\alpha}(B_1(x)\times ( t-1, t))}\leq C|\bfa|e^{-\lambda'_k t}\phi(x)^{(1-\varepsilon)\frac{\lambda'_k}{\lambda}},\quad(x, t)\in \Si\times\mathbb R_-, \quad k=1,2,\cdots, L.
\end{equation}
Then we need to show 
\begin{align}
\Vol(M_{-\infty})- \Vol(M_0)=\int_{-\infty}^0\int_{M_ t} H^2 dvol_{M_ t}d t<\infty,
\end{align}
and thus the volume relation is valid conceptually.

In fact, recall that
\begin{align}
\pa_ t u = -VH,
\end{align}
where $V:= \inn{\nu}{\hat\nu}^{-1}\geq1$. Thus, on each end $M_{t,i}$ of $M_t$,
\begin{equation}\label{H squ 1}
\begin{aligned}
     \int_{-\infty}^0\int_{M_{t,i}}  H^2 dvol_{M_ t}d t&\leq \int_{-\infty}^0\int_{\mathbb R^n} u_ t^2 \det(\hat{g}_{ij})dy d t
     \\
     &\leq\int_{-\infty}^0\int_{\mathbb R^n} u_ t^2 \det (g_{ij}+2uh_{ij}+u^2 h_i^kh_{kj}+ \D_iu\D_ju)dy d t.  
\end{aligned}
\end{equation}
Since $g_{ij}$ converges to the standard metric in $C^{3,\alpha}$ as $|y|\to\infty$ by Remark \ref{A bound g alpha}, the distance $\rho(y)$ on $\Si$ and $|y|$ on $\mathbb{R}^n$ are equivalent. By the exponential decay of $\phi$ which says that $0<\phi(x)<Ce^{-c(\Si)\rho(x)}$ for some $c(\Si)>0$, we have on each end of $M_ t$, 
\begin{equation}
    \begin{aligned}
\phi^{2(1-\varepsilon)\frac{\lambda'_k}{\lambda}}(y)
\leq e^{-2(1-\varepsilon)c(\Si)\frac{\lambda'_k}{\lambda}\rho(y)}
    \leq Ce^{-\varepsilon'_k|y|},
    \end{aligned}
\end{equation}
for some $\varepsilon_k'>0$. 

Therefore, (\ref{H squ 1}) and \eqref{uk alpha} imply that
\begin{equation}
    \begin{aligned}
        \int_{-\infty}^0\int_{M_{t,i}}  H^2 dvol_{M_ t}d t
    \leq &C|\bfa|^2\int_{-\infty}^0\int_{\mathbb{R}^n} \sum_{k=1}^{L}e^{-2\lambda'_k t}e^{-\varepsilon'_k|y|}dy  d t    <\infty.
    \end{aligned}
\end{equation}
 The finiteness of total curvature can be derived similarly from \eqref{uk alpha} and \eqref{hatA}.   
\end{proof} 

\subsection{Mean convexity and asymptotic decay}
The next step is to prove the positivity of $u$ and $-\hat H$ if $P_-u=a_1e^{-\lambda_1 t}\phi_1(x)$ with $a_1>0$. 
\begin{theorem}\label{thm:speed}
Let $\mathbf{a}=(a_1,0,\ldots,0)$ with $a_1>0$. If $a_1 <\e_1$ for some $\e_1(\Si)>0$ small, then the solution $u=u(\mathbf{a})$ constructed in \eqref{thm:exist} and the mean curvature $-H$ of $M_t$ both are positive on $\Si\times \mathbb R_-$. Moreover, %the asymptotic behavior of \(u(\cdot, t)\), $ t\in (-\infty,0]$, is identical to that of the first eigenfunction \(\phi\), i.e.
\begin{align}
    \|u-a_1e^{-\lambda_{1} t}\phi\|_{C^{2,\alpha}(B_1(x)\times ( t-1, t))}\leq C a_1^2 e^{-2\lambda_{1} t}\phi(x),
    \quad \text{in } \Si\times \mathbb{R}_-.
\end{align}
\end{theorem}

\begin{proof}
By our construction of $u_R$ and \eqref{wr schau phi} with $l+1=1$, $m=1$, we have for some $C$ independent of $R>R_0$ and $(x,t)\in\Sigma\times\mathbb R_-$,
\begin{align}
    u_R=&a_1e^{-\lambda_{R,1} t}\phi_{R,1}+w_R\geq a_1e^{-\lambda_{R,1} t}\phi_{R,1}-Ca_1^2e^{-2\lambda_{R,1} t}\phi(x), \quad \text{in } \Si\times \mathbb R_-.
\end{align}
Taking $R\ra \infty$, we obtain that if $a_1\ll 1$, we get 
\begin{align}\label{eq:u>0R}
    u
    \geq (a_1-C a_1^2e^{-\lambda_1 t})e^{-\lambda_{1} t}\phi>0 \qu\text{in } \Si\ti \mathbb R_-.
\end{align}

To show that the mean curvature has a sign, by differentiating $u_R$ with respect to $ t$, we see that
\begin{align}\label{eq:u_tR}
\pa_ t u_R&=(-\la_{R,1})a_1e^{-\lambda_{R,1} t}\phi_{R,1}+\partial_ t w_R\geq -\lambda_{R,1}a_1e^{-\lambda_{R,1} t}\phi_{R,1}-Ca_1^2e^{-2\lambda_{R,1} t}\phi(x) \text { in }\Si\times\mathbb R_-.
\end{align}
by \eqref{wr schau phi}. Taking the limit $R\ra \infty$, we obtain 
\begin{align}
\pa_ t u  \ge (-\la_1a_1 -Ca_1^2e^{-\lambda_{1} t})e^{-\la_1 t} \phi>0\text { in }\Si\times\mathbb R_-.
\end{align}
if we choose $a_1<\e_1\ll 1$. Thus $H=\tfrac{-\partial_ t u}{V}<0$ since $V>1$.

For the last assertion, we observe that
\begin{align}
    \|u_R-a_1e^{-\lambda_{R,1} t}\phi_{R,1}\|_{C^{2,\alpha}_{1,1}}= \|w_R\|_{C^{2,\alpha}_{1,1}}\leq C a_1^2 e^{-2\lambda_{R,1} t}\phi(x)
    \quad \text{in } \Si_R\times \mathbb R_-,
\end{align}
by \eqref{wr schau phi}. Taking $R\ra \infty$, we get 
\begin{align}
    \|u-a_1e^{-\lambda_{1} t}\phi\|_{C^{2,\alpha}_{1,1}}\leq C a_1^2 e^{-2\lambda_{1} t}\phi(x)
    \quad \text{in } \Si\times \mathbb R_-.
\end{align}
%and using the exponential decay of $\phi$ ({\red see Zhang Q.? }), the conclusion follows.
\end{proof}

\bigskip

\subsection*{Acknowledgments}
KC has been supported by the KIAS Individual Grant MG078902, an Asian Young Scientist Fellowship, and the National Research Foundation(NRF) grant funded by the Korea government(MSIT) (RS-2023-00219980).
JH has been supported by the KIAS Individual Grant MG088501.
TL has been supported by the NRF grant funded by the Korea government (MSIT) (No. RS-2023-00211258).


\begin{thebibliography}{99}
\bibitem{A}S. Agmon, Shmuel {\em Lectures on exponential decay of solutions of second-order elliptic equations: bounds on eigenfunctions of N-body Schrödinger operators}. Math. Notes, 29 Princeton University Press, Princeton, NJ; University of Tokyo Press, Tokyo, 1982.
%\bibitem{agmon}S. Agmon, {\em Bounds on exponential decay of eigenfunctions of Schrödinger operators.Schr\"odinger operators} (Como, 1984), 1–38. Lecture Notes in Math., 1159, Springer-Verlag, Berlin, 1985.
%\bibitem{AW}S. Altschuler, L. Wu, {\em Translating surfaces of the non-parametric mean curvature flow with prescribed contact angle}. Calc. Var. Partial Differential Equations, {\bf 2} (1994), no.1, 101–111. 
%\bibitem{Ag}S. Angenent, {\em Shrinking doughnuts}. Nonlinear diffusion equations and their equilibrium states, 3 (Gregynog, 1989), 21-38. Progr. Nonlinear Differential Equations Appl., 7 Birkhäuser Boston, Inc., Boston, MA, 1992.
\bibitem{ADS19}S. Angenent, P. Daskalopoulos, N. Sesum, { \em Unique asymptotics of ancient convex mean curvature flow solutions}. J. Differential Geom. {\bf 111} (2019), no.3, 381-455.
\bibitem{ADS20}S. Angenent, P. Daskalopoulos, N. Sesum, {\em Uniqueness of two-convex closed ancient solutions to the mean curvature flow}. Ann. of Math. (2) {\bf 192} (2020), no.2, 353-436.
\bibitem{AY}S. Angenent, Q. You, {\em Ancient solutions to curve shortening with finite total curvature}. Trans. Amer. Math. Soc. {\bf 374} (2021), no.2, 863-880.
%\bibitem{BW}J. Bernstein, L. Wang, {\em A topological property of asymptotically conical self-shrinkers of small entropy}. Duke Math. J. {\bf 166} (2017), no.3, 403-435.
\bibitem{BC19}S. Brendle, K. Choi, {\em Uniqueness of convex ancient solutions to mean curvature flow in $\mathbb{R}^3$}. Invent. Math.{\bf 217} (2019), no.1, 35-76.
\bibitem{BC21}S. Brendle, K. Choi, {\em Uniqueness of convex ancient solutions to mean curvature flow in higher dimensions}. Geom. Topol. {\bf 25} (2021), no.5, 2195-2234.
\bibitem{CCMS20}O. Chodosh, K. Choi, C. Mantoulidis, F. Schulze, {\em Mean curvature flow with generic initial data}, arXiv preprint arXiv:2003.14344,  (2020), to appear in Invent. Math.
\bibitem{CCS23}O. Chodosh, K. Choi, C. Mantoulidis, F. Schulze,  {\em Mean curvature flow with generic initial data II,} arXiv preprint arXiv:2302.08409, (2023).
\bibitem{CM}O. Chodosh, D. Maximo, {\em On the topology and index of minimal surfaces}, J. Differential Geom. {\bf 104} (2016), no. 3, 399-418.
\bibitem{CDDHS22}B. Choi, P. Daskalopoulos, W. Du, R. Haslhofer, N. Sesum, {\em Classification of bubble-sheet ovals in $\mathbb{R}^4$}.
arXiv preprint arXiv:2209.04931, (2022). 
\bibitem{CHH21}K. Choi, R. Haslhofer, O. Hershkovits, {\em A nonexistence result for wing-like mean curvature flows in $\mathbb R^4$}. arXiv preprint arXiv:2105.13100, (2021), to appear in Geom. Topl. 
\bibitem{CHH23}K. Choi, R. Haslhofer, O. Hershkovits, {\em Classification of noncollapsed translators in {$\mathbb R^4$}}. Camb. J. Math. {\bf 11}, 3 (2023), 563-698.
\bibitem{CHHW}K. Choi, R. Haslhofer, O. Hershkovits, B. White, {\em Ancient asymptotically cylindrical flows and applications}. Invent. Math. {\bf 229} (2022), no.1, 139-241.
\bibitem{CHH}K. Choi, R. Haslhofer, O. Hershkovits, {\em Ancient low-entropy flows, mean-convex neighborhoods, and uniqueness}. Acta Math. {\bf 228} (2022), no.2, 217-301.
\bibitem{ChoiM}K. Choi, C. Mantoulidis, {\em Ancient gradient flows of elliptic functionals and Morse index}. Amer. J. Math.{\bf 144} (2022), no.2, 541–573.
\bibitem{CS1}K. Choi, L. Sun,{\em Ancient finite entropy flows by powers of curvature in $\mathbb{R}^2$}. Nonlinear Anal. {\bf 216} (2022), Paper No. 112673, 19 pp.
\bibitem{DHS}P. Daskalopoulos, R. Hamilton, N. Sesum, {\em Classification of compact ancient solutions to the curve shortening flow}. J. Differential Geom. {\bf 84}(2010), no.3, 455-464.
\bibitem{DH24}W. Du, R. Haslhofer, {\em Hearing the shape of ancient noncollapsed flows in {$\mathbb R^4$}}. Comm. Pure Appl. Math. {\bf 77}, 1 (2024), 543-582.
\bibitem{F}D. Fischer-Colbrie, {\em On complete minimal surfaces with finite Morse index in three-manifolds}, Invent. Math. {\bf 82} (1985), no. 1, 121-132.
\bibitem{GT}D. Gilbarg, N. S. Trudinger, {\em Elliptic partial differential equations of second order}, Classics Math, Springer-Verlag, Berlin, 2001.
%\bibitem{Han-Lin}Q. Han, F. Lin, {\em Elliptic partial differential equations}. Second edition Courant Lect. Notes Math, 1, Courant Institute of Mathematical Sciences, New YorkAmerican Mathematical Society, Providence, RI, 2011.
\bibitem{H}Y. Han, {\em Ancient mean curvature flows from minimal hypersurfaces}, arXiv Preprint, arXiv:2311.15278v2, (2023). %\href{https://arxiv.org/abs/2311.15278} {arXiv:2311.15278v2}, (2023).
\bibitem{HH}R. Haslhofer, O. Hershkovits, {\em Ancient solutions of the mean curvature flow}. Comm. Anal. Geom. {\bf 24}(2016), no.3, 593-604.
\bibitem{HS}P. D. Hislop, I. M. Sigal, {\em Introduction to spectral theory: with applications to Schr\"odinger operators}, Appl. Math. Sci., 113, Springer-Verlag, New York, 1996.
%\bibitem{HIMW}D. Hoffman, T. Ilmanen, F. Martin, B. White, {\em Graphical translators for mean curvature flow}. Calc. Var. Partial Differential Equations, {\bf 58} (2019), no.4, Paper No. 117
%\bibitem{HMW}D. Hoffman, F. Martin, B. White, {\em Scherk-like translators for mean curvature flow}.J. Differential Geom. {\bf 122} (2022), no.3, 421-465.
%\bibitem {KKM}N. Kapouleas, S. Kleene, N. Møller, {\em Mean curvature self-shrinkers of high genus: non-compact examples}. J. Reine Angew. Math. {\bf 739} (2018), 1-39.
\bibitem{I}T. Ilmanen, {\em Elliptic regularization and partial regularity for motion by mean curvature}. Mem. Amer. Math. Soc.{\bf 108}(1994), no.520, x+90 pp.
\bibitem{JM}L. Jorge, W. H. Meeks III, {\em The topology of complete minimal surfaces of finite total Gaussian curvature}. Topology {\bf 22} (1983), no.2, 203-221.
\bibitem{Li}C. Li, {\em Index and topology of minimal hypersurfaces in $\mathbb{R}^n$} Calc. Var. Partial Differential Equations {\bf 56} (2017), no.6, Paper No. 180, 18 pp.
\bibitem{MS}J. Metzger, F. Schulze, Felix, {\em No mass drop for mean curvature flow of mean convex hypersurfaces}. Duke Math. J.{\bf 142}(2008), no.2, 283-312.
%\bibitem{MiSi}J. Michael, L. Simon, {\em Sobolev and mean-value inequalities on generalized submanifolds of $\mathbb R^n$}. Comm. Pure Appl. Math. {\bf 26} (1973), 361-379.
\bibitem{MP}A. Mramor, A. Payne, {\em Ancient and eternal solutions to mean curvature flow from minimal surfaces}. Math. Ann. {\bf 380}(2021), no.1-2, 569–591.
\bibitem{M}A. Mramor, {\em A classification result for eternal mean convex flows of finite total curvature type}, arXiv Preprint, arXiv: 2403.12020v1, (2024). %\href{https://arxiv.org/pdf/2403.12020.pdf}{arXiv: 2403.12020v1}, (2024).
\bibitem{Na}S. Nayatani, {\em Morse index and Gauss maps of complete minimal surfaces in Euclidean 3-space}, Comment. Math. Helv. {\bf 68} (1993), no. 4, 511-537.
%\bibitem{N1}X. Nguyen, {\em Construction of complete embedded self-similar surfaces under mean curvature flow. Part I}. Trans. Amer. Math. Soc. {\bf 361}(2009), no.4, 1683–1701.
% \bibitem{N tri}X. Nguyen, {\em Translating tridents}.Comm. Partial Differential Equations, {\bf 34}(2009), no.1-3, 257-280.
% \bibitem{N2}X. Nguyen, {\em Construction of complete embedded self-similar surfaces under mean curvature flow. Part II}. Adv. Differential Equations {\bf 15} (2010), no.5-6, 503-530.
% \bibitem{N3}X. Nguyen, {\em Construction of complete embedded self-similar surfaces under mean curvature flow, part III}. Duke Math. J. {\bf 163}(2014), no.11, 2023–2056.
% \bibitem{N dou}X. Nguyen, {\em Doubly periodic self-translating surfaces for the mean curvature flow}. Geom. Dedicata. {\bf 174}(2015), 177-185. 
\bibitem{Osser 1}R. Osserman, {\em Global properties of minimal surfaces in $E^3$ and $E^n$}. Ann. of Math., {\bf 80(2)}: 340-364, 1964.
\bibitem{P}A. Payne. {\em Mass drop and multiplicity in mean curvature flow},  arXiv Preprint, arXiv:2009.14163v1, (2020).  % \href{https://arxiv.org/abs/2009.14163}{arXiv:2009.14163v1}, (2020).
\bibitem{S}R. M. Schoen, {\em Uniqueness, symmetry, and embeddedness of minimal surfaces}.J. Differential Geom. {\bf 18} (1983), no.4, 791-809.
\bibitem{SY}R. M. Schoen, S.-T. Yau, {\em Lectures on differential geometry}, Conf. Proc. Lecture Notes Geom. Topology, I, International Press, Cambridge, MA, 1994.
\bibitem{Tysk fini}J. Tysk, {\em Finiteness of index and total scalar curvature for minimal hypersurfaces}. Proc. Amer. Math. Soc. {\bf 105} (1989), no.2, 429-435.
%\bibitem{Wa}L. Wang, {\em Uniqueness of self-similar shrinkers with asymptotically conical ends}. J. Amer. Math. Soc. {\bf 27}(2014), no.3, 613–638.
%\bibitem{W}B. White, {\em The nature of singularities in mean curvature flow of mean-convex sets}. J. Amer. Math. Soc. {\bf 16} (2003), no.1, 123-138.
% \bibitem{wei}G. Wei, {\em Comparison Geometry for Ricci Curvature}, \href{https://web.math.ucsb.edu/~wei/paper/08summer-lecture.pdf}{notes on line}.



\end{thebibliography}
\end{document}